\documentclass[12pt,reqno]{amsart}
\usepackage{hyperref}

\usepackage{amsfonts, amssymb,amsmath,amsthm,enumerate}
\usepackage[margin=1.5in]{geometry}

\RequirePackage{mathrsfs} \let\mathcal\mathscr

\newtheorem{theorem}{Theorem}[section]
\newtheorem{lemma}[theorem]{Lemma}

\newtheorem{corollary}[theorem]{Corollary}

\theoremstyle{definition}
\newtheorem*{ack}{Acknowledgements}
\newtheorem{rem}[theorem]{Remark}

\newcommand{\PP}{\mathbb{P}}
\renewcommand{\AA}{\mathbb{A}}
\newcommand{\A}{\mathbf{A}}

\newcommand{\NN}{\mathbb{Z}_{>0}}
\newcommand{\QQ}{\mathbb{Q}}
\newcommand{\RR}{\mathbb{R}}

\makeatletter
\def\imod#1{\allowbreak\mkern10mu({\operator@font mod}\,\,#1)}

\makeatother

\newcommand{\xx}{\textbf{x}}

\newcommand{\yy}{\textbf{y}}
\newcommand{\ww}{\textbf{w}}
\newcommand{\zz}{\textbf{z}}

\newcommand{\rrr}{\textbf{r}}

\newcommand{\ff}{\textbf{f}}
\newcommand{\ggs}{\textbf{g}}
\newcommand{\uu}{\textbf{u}}

\newcommand{\beq}{\begin{equation}}
\newcommand{\eeq}{\end{equation}}
\newcommand{\beqq}{\begin{equation*}}
\newcommand{\eeqq}{\end{equation*}}
\newcommand{\bal}{\begin{align}}
\newcommand{\eali}{\end{align}}
\newcommand{\ball}{\begin{align*}}
\newcommand{\ealii}{\end{align*}}
\newcommand{\eps}{\varepsilon}

\numberwithin{equation}{section}
\newcommand{\ve}{\varepsilon}
\DeclareMathOperator{\Vol}{vol}
\renewcommand{\leq}{\leqslant}

\renewcommand{\geq}{\geqslant}

\renewcommand{\d}{\mathrm{d}}
\DeclareMathOperator{\Mod}{mod} 
\renewcommand{\bmod}[1]{\,(\Mod{#1})}
\newcommand{\N}{\mathrm{N}}

\begin{document}
\title[Counting rational points on del Pezzo surfaces]{Counting rational points on del Pezzo surfaces with a conic bundle structure}
\author{T.D. Browning}
\address{School of Mathematics\\
University of Bristol\\ Bristol\\ BS8 1TW}
\email{t.d.browning,mike.swarbrickjones@bristol.ac.uk}
\author{M. Swarbrick Jones}

\begin{abstract}
For any number field $k$, upper bounds are established for the number of $k$-rational points of bounded height on non-singular del Pezzo 
surfaces defined over $k$, which are equipped with suitable  conic bundle structures over $k$.
\end{abstract}


\maketitle
\setcounter{tocdepth}{1}
\tableofcontents

\thispagestyle{empty}

\section{Introduction}\label{intro}

Let $k$ be a number field.  A del Pezzo surface $X$ over $k$ is a non-singular projective surface defined over $k$, with ample anticanonical 
divisor $-K_X$. The degree of $X$ is defined to be $d= (-K_X)^2$. 
In this paper we will be concerned with upper bounds for the number of  $k$-rational points of bounded height on del Pezzo surfaces of small degree. The arithmetic of del Pezzo surfaces becomes harder to understand as $d$ decreases. 
For $d\in \{2,3,4\}$ they admit the following classical description:
\begin{itemize}
	\item[---] an intersection of two quadrics in $\mathbb{P}^4$ when $d=4$;
	\item[---] a cubic surface in $\mathbb{P}^3$ when $d=3$;
	\item[---] a double cover of $\mathbb{P}^2$ branched over a smooth quartic plane curve when $d=2$.
\end{itemize}
Given a del Pezzo surface $X$ of degree $d$, let $U\subset X$ be the
Zariski open set obtained by deleting from $X$ the finite set of exceptional curves of the first kind.
Let 
$$
N(U,k,B) = \# \{ x \in U(k): H_{-K_X}(x) \leq B  \} , 
$$ 
where $H_{-K_X}$ is the anticanonical height function on the set $X(k)$ of $k$-rational points on $X$.
Our motivation is a simple form of the Batyrev--Manin conjecture \cite{B-M}, which implies that we should have 
\beq\label{DPest} N(U,k,B)=O_{\eps, X}( B^{1+ \eps}), \eeq
for any $\eps>0$.
Throughout this paper, unless otherwise indicated, we shall follow the convention that any implied constant is allowed to depend at most 
upon the number field $k$, with any further dependence explicitly indicated. 
In \eqref{DPest}, for example, the implied constant is allowed to depend on $X$ and  the choice of $\eps$, in addition to  $k$.

Recall that 
a conic bundle surface over $k$ is defined to be a non-singular  projective surface $S$
defined over $k$, which is equipped with a dominant $k$-morphism
$S\rightarrow \mathbb{P}^1$, all of whose fibres are conics.
We shall focus our attention on del Pezzo surfaces of degree $d$ which are also conic bundle surfaces. 
When no such restriction is made on the del Pezzo surface, the best general bound we have is due to 
Salberger \cite{bonn}. Working in the special case $k=\mathbb{Q}$, he has established the estimate
\begin{equation}\label{eq:salberger}
N(U,\mathbb{Q},B) =O_{\eps,X}(B^{3 / \sqrt{d}+\eps}),
\end{equation}
for any $\eps>0$.

Let us first consider the case of degree $4$ del Pezzo surfaces $X\subset \mathbb{P}^4$ defined over $k$. 
In work presented at the conference ``Higher dimensional varieties and rational points'' at Budapest in 2001, Salberger noted that one can get better bounds for 
$N(U,k,B)$ when $X$
contains a non-singular  conic over $k$, in which case 
it  has a conic bundle structure over $k$.  For such surfaces 
he established \eqref{DPest} when  $k=\mathbb{Q}$. The following result generalises this to arbitrary number fields. 

\begin{theorem}\label{dp4thm}  Let $\eps>0$ and let $X\subset \PP^4$ be a del Pezzo surface of degree $4$  over $k$,  containing a non-singular  conic defined over $k$.  Then 
$$
N(U,k,B)=O_{\eps,X}( B^{1 + \eps}).$$ 
The implied constant is ineffective.
\end{theorem}

All of the implied constants in our results about del Pezzo surfaces are ineffective. This arises from an application
  of the Thue--Siegel--Roth theorem over number fields 
  \cite[\textsection 7, Thm.~1.1]{Lang} (see Remark \ref{remark} for an indication of how effectivity can be recovered). 

In the case $k=\mathbb{Q}$, 
de la Bret\`{e}che and Browning \cite{BreBro} have obtained an asymptotic formula for 
$N(U,\mathbb{Q},B)$, as $B\rightarrow\infty$,  for a particular del Pezzo surface of degree $4$ with a conic bundle structure over $\QQ$. In  general the best bound available is given by \eqref{eq:salberger}, although one can do better if one is willing to 
assume a standard rank hypothesis for elliptic curves over $\mathbb{Q}$ (see \cite[Section~7.3]{Timbook}).

According to Iskovskih's  $k$-birational classification 
 \cite{Isko}, there are two possible classes of degree $4$ conic bundle surface defined over $k$. When the anticanonical divisor is ample  one has a del Pezzo surface of degree $4$, as considered in  Theorem \ref{dp4thm}. When the anticanonical divisor is not ample, on the other hand, one obtains a  Ch\^{a}telet surface  to which one can associate 
an analogous counting function $N(U,k,B)$. In this setting one still expects \eqref{DPest} to hold and 
Browning \cite{TimChatelet} has established this when $k=\mathbb{Q}$. 
Although we choose not to do so here, it is possible to use the results in this paper to 
extend 
this work to Ch\^atelet surfaces defined over arbitrary number fields.

We now turn to del Pezzo surfaces of degree $3$ over $k$. These arise as cubic surfaces $X$ in $\PP^3$.
We do not know of a single non-singular cubic surface for which \eqref{DPest} has been proved. 
Cubic surfaces admit a conic bundle structure over $k$ when one of the $27$ lines contained in the surface is defined over $k$.
The best bounds that we have for $N(U,k,B)$ arise when stronger 
hypotheses are placed on the configuration of lines in the surface. The following result is due to 
Broberg \cite{Broberg} and is 
a generalisation of the  case $k=\QQ$  handled by  Heath-Brown \cite{HB}.

\begin{theorem}\label{dp3thm}  
Let $\eps>0$ and let $X\subset \PP^3$ be a del Pezzo surface of degree $3$ over $k$, containing three coplanar lines defined over $k$.  Then 
$$
N(U,k,B)=O_{\eps,X}( B^{4/3 + \eps}).
$$ 
The implied constant is ineffective. 
\end{theorem}

Note that the implied constant is actually effective in 
\cite{HB} and \cite{Broberg}.
We will provide our own  proof of Theorem~\ref{dp3thm},  since
our argument is simpler than that appearing in \cite{Broberg}, albeit
at the expense of effectivity in the implied constant. 
In the case $k=\QQ$, the best general bound is given by 
\eqref{eq:salberger} with $d=3$. 
There is also  further work of Heath-Brown \cite{HB2} when $k=\QQ$, 
which shows that the estimate in Theorem \ref{dp3thm} holds for all 
cubic surfaces conditionally on the rank hypothesis mentioned previously.

Much less is known about the arithmetic of del Pezzo surfaces of degree $2$ over $k$. These may be embedded in weighted projective space $\PP(2,1,1,1)$ via an equation of the form
\begin{equation}\label{eq:dp2}
t^2=f(x_1,x_2,x_3),
\end{equation}
where $f\in k[x_1,x_2,x_3]$ is a non-singular  form of degree $4$. 
Taking $U\subset X$ to be the complement of the 56 exceptional curves, 
it was shown by Broberg \cite[Thm.~2]{Broberg2} that $N(U,\QQ,B)=O_{\ve,X}(B^{9/4+\ve})$ for any $\ve>0$. This is improved upon by \eqref{eq:salberger},  but both bounds are rather far from the expectation in \eqref{DPest}.
In the spirit of the previous results it is possible to exploit conic bundle structures.

\begin{theorem}\label{dp2thm}  
Let $\eps>0$ and let $X\subset \PP(2,1,1,1)$ be a del Pezzo surface of degree $2$, 
containing a non-singular conic defined over $k$.
Then 
$$
N(U,k,B)=O_{\eps,X}( B^{2 + \eps}).
$$ 
The implied constant is ineffective.
\end{theorem}

One can do better  when  $k=\QQ$ and 
one assumes that all of the 56 exceptional curves in $X$  
are defined over $\QQ$. In this case, as announced by 
Salberger
at the conference ``G\'eom\'etrie arithm\'etique et vari\'et\'es rationnelles'' at Luminy in 2007, 
one has the  sharper bound 
$
N(U,\QQ,B)=O_{\ve,X}(B^{11/6+\ve}).
$

We  proceed
 to indicate the contents of this paper. 
We shall use the underlying conic bundle structures to prove Theorems~\ref{dp4thm}--\ref{dp2thm}, closely following the notation and framework developed by Broberg \cite{Broberg}. 
In Section \ref{conicsection} we 
recall some basic facts from algebraic number theory and 
present our main technical result,
Theorem \ref{conicbundleest}, 
  from which our results on del Pezzo surfaces are deduced in Section 
  \ref{rationalexamples}. This is concerned with counting $k$-rational points of bounded height on certain ``conic bundle torsors'' and its  
  proof  hinges upon two further results: Theorem \ref{qmain} and 
Theorem 
\ref{sumestimate}. The first of these involves
counting $k$-rational points of bounded height on non-singular conics defined over $k$, 
which needs to be done uniformly with respect to the coefficients of the underlying equation.
This is likely to be of independent interest and is proved in 
 Section \ref{ternary}. The second is concerned with 
 a certain average involving binary forms over $k$ 
and  is proved in Section~\ref{binarysum}.

\begin{ack}
The authors are grateful to  Per Salberger for comments 
on an earlier draft of this paper. 
While working on this paper the first author 
was  supported by ERC grant \texttt{306457} and the second author 
was supported by an EPSRC doctoral training grant. 
\end{ack}

\section{Counting points on conic bundle torsors}\label{conicsection}

\subsection{Algebraic number theory}
We begin by recalling some basic notation and facts 
concerning our number field $k$.  Let $d=[k:\mathbb{Q}]$ and let $\mathfrak{o}$ be the ring of integers of $k$.  
We denote by  $\Omega$ the set of places of $k$.
 We let $s_k$ denote the number of infinite places of $k$.  For any  $\nu\in \Omega$, we let $\mu$ be its restriction to $\mathbb{Q}$ and put $d_\nu = [k_\nu : \mathbb{Q}_\mu]$.  The absolute value $|\cdot|_\nu$ on $k$ is the one which induces the normal absolute value on $\mathbb{R}$ if $\nu\mid \infty$ and the $p$-adic absolute value if $\nu\mid p$.  
 The normalised
 absolute value is 
 $$
 \| \cdot \|_\nu = |\cdot |^{d_\nu}_\nu.
 $$ 
 We denote by  $\N_k(\mathfrak{a})=[\mathfrak{o}:\mathfrak{a}]$ 
 the ideal norm for any fractional ideal $\mathfrak{a}$ of $\mathfrak{o}$. We also have 
 $\N_k(\langle \alpha \rangle)=|N_{k/\QQ}(\alpha)|$, where $\langle \alpha \rangle$ is the principal ideal generated by any $\alpha\in k^*.$  In this case we write $\N_k(\alpha)$ for short. 
Recall that for any $x \in k^*$ we have 
\beq\label{productrule}  \prod_{\nu\mid\infty }\| x \|_\nu = \N_k(x) \quad \text{ and } \quad \prod_{\nu\in \Omega}\| x \|_\nu = 1,
\eeq 
the second equation being the {\em product formula}.

There is a well-defined height function 
$H_k : \mathbb{P}^n(k) \mapsto \mathbb{R}_{\geq 1}$, given by
$$
[x_0 , \dots , x_n] \mapsto \prod_{\nu\in \Omega} \sup_{0\leq i\leq n}  \| x_i \|_\nu  
= \frac{1}{\N_k(\langle x_0 , \dots , x_n \rangle)} \prod_{\nu\mid\infty} \sup_{0\leq i\leq n} \| x_i \|_\nu,
$$
where   $\langle x_0 , \dots , x_n \rangle $ denotes the $\mathfrak{o}$-span of $x_0 , \dots , x_n\in k^*$. 
We define a further distance function 
$\|\cdot \|_\star:k\rightarrow \RR_{\geq 0}$ via
$$
\|x \|_{\star} = \sup_{\nu\mid\infty} \| x \|_\nu.
$$  
For any $x \in k$ 
it is clear that 
\beq\label{Normupper}   \N_k(x) \leq \| x \|^{s_k}_\star.
\eeq
If $\textbf{x} = (x_1, \dots , x_n) \in k^n$, then
this distance function extends via
$$
\|\textbf{x} \|_{\star} =  \sup \limits_{1\leq i\leq n} \| x_i \|_{\star}=
 \sup_{\substack{
 1\leq i\leq n\\
 \nu\mid\infty}} \| x_i \|_\nu.
$$

Over $\mathbb{Q}$, any point $
x
 \in \mathbb{P}^{n}(\mathbb{Q})$ has a representative $\xx=(x_0 , \dots , x_n) \in \mathbb{Z}^{n+1}$ such that $\gcd (x_0 , \dots , x_n) = 1$, which easily allows one to  take  precisely one element from each equivalence class.  Over $k$, an analogue arises by first fixing once and for all a set of integral ideals $ \mathfrak{a}_1, \dots , \mathfrak{a}_h$
 representing classes in the ideal class  group.  Then any $x \in \mathbb{P}^{n}(k)$ has a representative in  coordinates $\xx\in \mathfrak{o}^{n+1}$ such that the $\mathfrak{o}$-span 
 $\langle x_0 , \dots , x_n \rangle$ is one of the ideals $\mathfrak{a}_i$.    A useful consequence of Dirichlet's unit theorem is the following standard result (see 
 \cite[Prop.~3]{Broberg} for a proof).

\begin{lemma}\label{important1}   Every point $x\in \mathbb{P}^n(k)$ has a representative $\xx\in \mathfrak{o}^{n+1}$ such that $\langle x_0, \dots , x_n \rangle = \mathfrak{a}_i$ for some $i\in \{1,\dots,h\}$, and 
\beq\label{Hsigma}
\|\xx\|_{\star}
  \ll_{n} H_k(x)^{1/{s_k}}.
\eeq 
\end{lemma}

According to our convention the implied constant in \eqref{Hsigma} is allowed to depend 
on $k$ in addition to $n$. In particular it is allowed to depend on the set of representative ideals $\mathfrak{a}_1,\dots,\mathfrak{a}_h$ that were fixed above.
We may now  define the sets
$$
Z'_{n+1}=
\bigcup_{1\leq i\leq h}
\left\{
(x_0, \dots , x_n) \in \mathfrak{o}^{n+1}: 
\langle x_0, \dots , x_n \rangle = \mathfrak{a}_i 
\right\}
$$
and 
$$
Z_{n+1}=
\left\{
(x_0, \dots , x_n) \in Z'_{n+1}: 
\mbox{\eqref{Hsigma} holds}
\right\}.
$$
Lemma~\ref{important1} implies that 
associated to any element of $\mathbb{P}^n(k)$ is an element of 
 $Z_{n+1}$.
Note, however, that elements of the latter set do not uniquely determine elements of the former. Nonetheless this is is sufficient for our purposes.  
It follows from   Lemma \ref{important1}  that
\beq\label{Hasymp} H_k(x)^{1/s_k} \leq  \| \xx \|_{\star} \ll H_k(x)^{1/s_k}, 
\eeq
for every $x\in \PP^n(k)$ and corresponding element $\xx \in Z_{n+1}.$

 \subsection{Conic bundle torsors}

Let  $f_{ij}\in k[u,v]$ be binary forms for $1 \leq i,j \leq 3$.  
Let $S_1\subset \AA^1\times \PP^2$ be given by the equation
$$
\sum_{i,j=1}^3f_{ij}(u,1) x_i x_j =0,
$$
and let 
$S_2\subset \AA^1\times \PP^2$ be given by 
$$
\sum_{i,j=1}^3f_{ij}(1,v) x_i x_j =0.
$$
We shall assume that 
 every principal $2 \times 2$ minor of the matrix $\mathbf{F}=(f_{ij} )$ is a binary form of even degree and, furthermore, 
 that 
$\Delta (u,v) = \det ( \mathbf{F})$ is separable. 

Let  $d_i$ be the degree of the cofactor of the diagonal element $f_{ii}$ in $\mathbf{F}$ (e.g. $d_1$ is the degree of the bottom right $2 \times 2$ minor).  We let $U_1\subset S_1$ be the open subset  given by $u \neq 0$, and we let $U_2\subset S_2$ be the open subset given by $v \neq 0$.  We obtain a conic bundle surface $S$ by 
 glueing $U_1$ and $U_2$ via the isomorphism
$$
( u; [x_1,x_2,x_3] ) \mapsto ( 1/u; [x_1 u^{-d_1/2},x_2 u^{-d_2/2},x_3 u^{-d_3/2}]).
$$ 
The morphisms  $S_i \rightarrow \PP^1$ 
given by $(u;[x_1 , x_2 , x_3 ] )\mapsto [u , 1]$ for $i=1$ and 
$(v;[x_1 , x_2 , x_3] )\mapsto [1 , v ]$ for $i=2$,
glue together to give a conic fibration $$\phi: S \rightarrow \PP^1.$$ 
Since $\Delta(u,v)$ is separable, it follows from 
 \cite[\S II.6.4, Prop.~1]{Shaf} that $S$ is non-singular.
The singular fibres of $\phi$ correspond to the roots of $\Delta (u,v)$.  
We define the degree of $S$ to be $(-K_S)^2 = 8 - r$, 
where $r$ is the number of singular fibres of $\phi$.

Consider
the variety $\mathcal{T} \subset  \mathbb{A}^2 \times \mathbb{P}^2 $ given by  \beq\label{bundleshape2}
\sum_{i,j = 1} ^3 f_{ij}(u,v) x_i x_j=0. 
\eeq  
We claim that  $\mathcal{T}$ is a torsor over $S$.
There is a morphism $\pi:\mathcal{T} \rightarrow S$ as follows.  
Let $(u,v; [x_1,x_2,x_3] )\in \mathcal{T}$.  
If $v \neq 0$ then  
$$( u/v;[x_1 v^{-d_1/2},x_2 v^{-d_2/2},x_3 v^{-d_3/2}]) \in S_1,
$$ 
while  if $u \neq 0$ then 
$$
(v/u;[x_1 u^{-d_1/2},x_2 u^{-d_2/2},x_3 u^{-d_3/2}]) \in S_2.
$$  
It is clear that 
$\mathbb{G}_m^2$ acts on $\mathcal{T}$ via 
$$
( \lambda, \mu) \mapsto (\mu,\mu;[\lambda \mu^{d_1 / 2}, \lambda \mu^{d_2 / 2}, \lambda \mu^{d_3 / 2}] ),
$$ 
and this acts on $\AA^2\times \mathbb{P}^2$ in the natural way.  This action  is free and transitive  on the fibres of $\pi$ and so $\mathcal{T}$ is indeed a torsor over $S$.
We shall henceforth refer to varieties of the shape \eqref{bundleshape2} as {\em conic bundle torsors}, whenever every principal  $2\times 2$ minor of $\mathbf{F}$ is a binary form of even degree and $\Delta(u,v)=\det(\mathbf{F})$ is separable.

\subsection{Counting rational points}

Let $\mathcal{T}$ be a conic bundle torsor, and let $\mathcal{T}_0\subset \mathcal{T}$ be the open subset on which $\Delta(u,v) \neq 0$.  
Let $\rrr \in (\mathbb{R}_{\geq 1 })^{s_k}$ be a vector with components $r_\nu$ for $\nu\mid \infty$. Define 
$$
L(\rrr)=
\left\{ x \in \mathfrak{o}: \mbox{$\| x \|_\nu \leq r_\nu$ for $\nu\mid \infty$}
\right\} .
$$  
Let  
$$
\| \rrr \| = \prod_{\nu\mid \infty} r_\nu
$$  
and put $\underline{\rrr} = (\rrr_1, \rrr_2, \rrr_3)$, for 
$\rrr_1,\rrr_2,\rrr_3 \in (\mathbb{R}_{\geq 1 })^{s_k}$.
For   given  $x \in \mathbb{P}^2(k)$, it will be a convenient abuse of notation to write
 $x \in L( \underline{\rrr})\cap Z'_3$ if there is a  representative $(x_1, x_2,x_3) \in Z'_3$
of $x$  such that $x_i \in L(\rrr_i)$ for $1\leq i\leq 3.$ 
In this case we will always associate a unique such representative.

For given $A \geq 1$ and $\underline{\rrr} = (\rrr_1, \rrr_2, \rrr_3)$, 
with $\rrr_1,\rrr_2,\rrr_3 \in (\mathbb{R}_{\geq 1 })^{s_k}$, we
define the counting function
$$
N_{\mathcal{T}_0}(A, \underline\rrr ) = \# \left\{(u,v) \in Z_2, ~x \in \mathbb{P}^2(k) : 
\begin{array}{l}
(u,v; x) \in \mathcal{T}_0(k) \\ 
A \leq 
H_k ([u,v]) 
< 2A \\
x \in L(\underline\rrr)  \cap Z'_3 
\end{array}{}
\right\}.
$$  
In Section \ref{rationalexamples} we shall show that the proof of Theorems \ref{dp4thm}--\ref{dp2thm} 
can essentially be reduced to special cases of the following general estimate.

\begin{theorem}\label{conicbundleest}  Let $\ve>0$ and let  $\mathcal{T}$ 
be a conic bundle torsor
of the shape \eqref{bundleshape2}, with $\deg \Delta(u,v)=n$. Then
$$ 
N_{\mathcal{T}_0}(A,\underline\rrr) \ll_{\eps,\mathcal{T}} A^{2+\eps} \left( 1 + \left(\frac{\| \rrr_1 \| \| \rrr_2 \| \| \rrr_3 \|}{A^n}\right)^{1/3}  \right).
$$  
The implied constant is ineffective.
\end{theorem}

We proceed to prove this theorem subject to some technical results which will be established in due course. The first, which should be of independent interest, concerns 
counting $k$-rational points on 
conics.

For a  matrix $\mathbf{M}\in \mathrm{GL}_3(\mathfrak{o})$, let $\Delta(\mathbf{M})$ and $\Delta_0 (\mathbf{M})$ be  the  ideals  generated by the determinant of $\mathbf{M}$ and the $2 \times 2$ minors of $\mathbf{M}$, respectively.  
Let $\tau$ be the usual divisor function on integral ideals.
Then we shall establish the following result in Section \ref{ternary}.

\begin{theorem}\label{qmain}   Let $Q$ be a non-singular ternary quadratic form with underlying matrix $\mathbf{M} \in \mathrm{GL}_3(\mathfrak{o})$.  Let $ \rrr_1, \rrr_2, \rrr_3\in (\RR_{\geq 1})^{s_k}$ be given.  There are 
$$
\ll  \left( 1 + \left( \frac{ \|  \rrr_1 \|  \|  \rrr_2 \|  \|  \rrr_3 \| \N_k (\Delta_0 (\mathbf{M}))^{3/2}}{\N_k (\Delta (\mathbf{M}))}   \right)^{1/3} \right) \tau(\Delta(\mathbf{M}))
$$ 
elements $x=[\xx] \in \mathbb{P}^2(k)$ such that $Q(\xx) = 0$ and  $x \in L(\underline\rrr) \cap Z'_3$. 
\end{theorem} 

This result generalises to arbitrary number fields a result of Browning and Heath-Brown
\cite[Cor.~2]{BrowningHB}.   It is important to note that the implied constant in this estimate depends at most on the field $k$, but is uniform in the coefficients of the quadratic form $Q$.

Theorem \ref{qmain} is a crucial ingredient in our proof of Theorem 
\ref{conicbundleest}. Indeed, if one takes $A=1$ in the latter result, then one obtains a version of Theorem~\ref{qmain} in which the implied constant is allowed to depend on the coefficients of $Q$. Alternatively, if $A$ is large compared to $\|\rrr_1\|\|\rrr_2\|\|\rrr_3\|$, then 
Theorem \ref{conicbundleest} shows that most conics in the family contribute very few points.

The second technical result we require concerns 
an average involving binary forms.  
The following result  will be established in Section \ref{binarysum} and is based on Lang's generalisation of the Thue--Siegel--Roth theorem to number fields.

\begin{theorem}\label{sumestimate}  
Let  $\eps >0$ and let $F(u,v) \in \mathfrak{o}[u,v]$ 
be a separable form 
of degree $n$. Then 
\begin{equation}\label{fc} 
\sum_{
\substack{ (u,v) \in \mathfrak{o}^2 \\ A^{1/s_{k}} \leq \|(u,v)\|_\star < 2A^{1/s_k} \\  F(u,v) \neq 0}}   
\frac{1}{(\N_k( F(u,v)))^{1/3} } \ll_{\eps, F} A^{2 - n/3 + \eps}.
\end{equation} 
The implied constant is ineffective.
\end{theorem}

We now have everything in place to establish Theorem \ref{conicbundleest}, conditionally on the  technical results.
Let $\mathcal{T}$ be a conic bundle  torsor of the shape \eqref{bundleshape2}.
We shall proceed  by counting the number of points on the fibres $C_{u,v}$ of $\mathcal{T}$ above $(u,v)\in \mathbb{A}^2$,  uniformly in $u,v$.
  Given $(u,v) \in \mathbb{A}^2$ such that $\Delta(u,v)\neq 0$, let $\mathbf{M}(u,v)$ be the matrix which produces  the conic $C_{u,v}$ above it.
We have  $\Delta(u,v)=\Delta(\mathbf{M}(u,v))$ and we put $\Delta_0(u,v)=\Delta_0(\mathbf{M}(u,v))$.
By the trivial estimate for the divisor function we have 
$$
\tau(\Delta(u,v))\ll_\ve (\N_k(\Delta(u,v)))^\ve.
$$
Likewise, since $\Delta(u,v)$ is separable, the proof of 
\cite[Lemma 7]{Broberg} shows that 
$\N_k(\Delta_0(u,v)) \ll_\mathcal{T} 1$, for $(u,v) \in Z_2$.  

For given $(u,v)\in Z_2$ such that $\Delta(u,v)\neq 0$, we put
$$
N(u,v,\underline\rrr) = \# \{x\in \mathbb{P}^2(k) \cap C_{u,v}:  x \in L(\underline\rrr) \cap Z'_3  \}.
$$
It  follows from  Theorem \ref{qmain} that
$$
N(u,v,\underline\rrr) \ll_{\ve, \mathcal{T}} 
\left(1 + \frac{R^{1/3}}{\N_k(\Delta(u,v))^{1/3}} \right) (\N_k(\Delta(u,v)))^{\eps},
$$ 
for any $\ve>0$, 
where  $R=\|  \rrr_1 \|  \|  \rrr_2 \|  \|  \rrr_3 \|$.
We easily obtain
\begin{align*}N_{\mathcal{T}_0}(A, \underline\rrr ) 
&\leq \sum_{\substack{ (u,v) \in Z_2 \\ A \leq H_k([u,v]) < 2A }} N(u,v, \underline\rrr) \\
& \ll_{\ve,\mathcal{T}} A^\ve \sum_{\substack{ (u,v) \in Z_2 \\ A \leq H_k([u,v]) < 2A }} \left( 1 + \frac{R^{1/3}}{\N_k( \Delta( u,v))^{1/3} } \right).
\end{align*}
Finally, recalling \eqref{Hasymp}, an application of 
Theorem  \ref{sumestimate} yields
$$
N_{\mathcal{T}_0}(A, \underline\rrr ) 
\ll_{\eps, \mathcal{T}} A^{2+2\eps} \left( 1 + \frac{R^{1/3}}{A^{n/3}}  \right).
$$
We complete the proof of Theorem 	\ref{conicbundleest} upon redefining the choice of $\eps$.

\begin{rem}\label{remark}
Although we require  it for Theorem \ref{conicbundleest}, it should be noted that Theorems \ref{dp4thm}--\ref{dp2thm} do not strictly require Theorem \ref{sumestimate}  and it is possible to recover effectivity with extra work. 
Instead one can make use of the fact that there are $O_{\ve,F}(G^{1/n}A)$ points $(u,v) \in Z_2$, with $H_k([u,v]) \leq A$ and $\N_k(F(u,v)) \leq G$ (see \cite[Lemma~9]{Broberg}).  For Theorem~\ref{conicbundleest}, however, this would only produce the desired contribution  when $G$ has order of magnitude $A^n$. For Theorems \ref{dp4thm} and \ref{dp3thm},
moreover,  using this alternative bound would  require  us to handle a subset of the fibres in a different manner (see  \cite[Prop.~7 and Lemma 8]{Broberg}).
\end{rem}

\section{Counting points on del Pezzo  surfaces}\label{rationalexamples}

\subsection{Heights and morphisms}\label{heights}  
We begin  with some general facts about the behaviour of heights under morphisms, as described by Serre
\cite[\S 2]{Serre1}.
Let  $X$ be 
a del Pezzo surface of degree $d\in \{2,3,4\}$ over a number field $k$ and let $U\subset X$ be the Zariski open subset obtained by deleting the exceptional curves.
For a morphism $g:X\rightarrow \mathbb{P}^{\ell}$ we write 
$H_g(x)=H_k(g(x))$, for any $x\in X(k)$, where $H_k$ is the  height on $\mathbb{P}^\ell(k)$.

Suppose we are given  morphisms
$$
f_i : X \rightarrow \mathbb{P}^1, \quad i=1,\dots,m.
$$
Let $f$ be the morphism 
$$ 
f : X \rightarrow \mathbb{P}^1\times \dots \times \mathbb{P}^1=(\mathbb{P}^1)^m,
$$ 
given by $(f_1, \dots , f_m)$.  
Now let
 $\psi: 
(\mathbb{P}^1)^m
\rightarrow \mathbb{P}^{2^m - 1}$
 be the multilinear Segre embedding, so then $\psi \circ f$ is a morphism.  
We shall assume that the morphism $\psi \circ f$ takes the shape 
$$
\psi \circ f(x) = [\phi_0(x) , \dots , \phi_{2^m - 1}(x)]
$$ 
on $U$, 
where $\phi_0,\ldots,\phi_{2^m-1}$ are  homogeneous polynomials of degree $e$ which do not simultaneously vanish on $X$.  

Let $ p = ([u_1 , v_1] , \dots , [u_m , v_m]) \in (\mathbb{P}^{1}(k))^m$ and  
$\psi(p) = [y_0 , \dots , y_{2^m - 1}]$. Then 
$$
\sup_{0\leq i\leq 2^{m}-1} \|y_i\|_\nu = \prod_{i=1}^m \sup \{ \|u_i \|_\nu , \|v_i \|_\nu \}
$$ 
for every $\nu\in \Omega$.  
Thus  we have 
$$
\sum_{i=1}^m \log H_{f_i} = \log H_{\psi \circ f}.
$$  
Furthermore, the functoriality of heights yields
$$
\log H_{\psi \circ f} = e \cdot \log H_k + O_{f,X}(1).
$$ 
It follows that there is an absolute constant $c_1=c_1(f,X)$ such that 
$$
\prod_{1\leq i\leq m}H_k(f_i(x)) \leq c_1^m  H_k(x)^e, 
$$
for any $x \in U(k)$.  Thus we have 
$$
H_k(f_i(x)) \leq c_1 H_k(x)^{e/m},
$$
for at least one $i\in \{1,\dots,m\}$.

In our work it will suffice to estimate  the counting function
$$
N(U,k,B) = \# \{ x \in U(k): H_{k}(x) \leq B \}.
$$
When $d=2$ and $X$ is embedded in $\mathbb{P}(2,1,1,1)$ via an equation of the form
\eqref{eq:dp2},  the height is taken to be  $H_k([x_1,x_2,x_3])$  on $\PP^2(k)$.   
When $d\in \{3,4\}$ it is taken to be the height on $\PP^d(k)$ that arises through the anticanonical embedding of $X$ in $\PP^d$.
We may now conclude as follows. 

\begin{lemma}\label{lem:height}
There exists a constant $c_1=c_1(f,X)>0$ such that 
$$
N(U,k,B) \leq \sum_{i=1}^m n_i(B),
$$
where
$$
n_i(B)=\#\left\{x \in U(k):  H_k(x) \leq B \mbox{ and } H_k(f_i(x)) \leq c_1 B^{e/m}\right\}.
$$ 
\end{lemma}

For the remainder of Section 
\ref{rationalexamples}
we shall allow all of the implied constants to depend implicitly on the number field $k$, the del Pezzo surface $X$ and the 
small parameter $\ve>0$ appearing in Theorems \ref{dp4thm}--\ref{dp2thm}.
Furthermore, in the light of Theorem \ref{conicbundleest}, we shall allow the implied constants to be ineffective.

\subsection{Proof of Theorem \ref{dp4thm}}  
After a change of variables we may assume that $X$ is given by 
\begin{align*} 
x_0 x_1 - x_2 x_3 &= 0 \\  
Q(x_0, x_1, x_2 , x_3) + a x_4^2
&= 0,
\end{align*}
for a  quadratic form $Q\in \mathfrak{o}[x_0,x_1,x_2,x_3]$ and a non-zero  element $a \in \mathfrak{o}$. 
Let $U\subset X$ be the subset obtained by deleting the 16 lines from  $X$.

We will consider two conic fibrations $f_1 , f_2 : U \rightarrow \mathbb{P}^1$, given by
\begin{align*}  f_1(x) &= 
\begin{cases}
[x_0 , x_2] & \text{if} \, \, (x_0 , x_2) \neq (0,0),  \\
[x_3 , x_1]  & \text{if} \, \, (x_3 , x_1) \neq (0,0),
\end{cases}\\
f_2(x) &= 
\begin{cases}
[x_0 , x_3]  & \text{if} \, \, (x_0 , x_3) \neq (0,0),  \\
[x_2 , x_1]  & \text{if} \, \, (x_2 , x_1) \neq (0,0).  \end{cases}
\end{align*}
Note that the definitions agree where the open sets intersect,  so that this is a well-defined morphism. Moreover, the open sets cover $X$ since there are no points on $X$ with $x_0 = x_1 = x_2 = x_3=0$.  
Define $f: X \rightarrow \mathbb{P}^1 \times \mathbb{P}^1$ to be the morphism given by $x \mapsto (f_1(x),f_2(x))$.  With the notation of Section \ref{heights}, we confirm that$$
\psi \circ f (x) =   [x_0, x_3, x_2 , x_1],
$$ 
for all $x \in U$. Thus we can take $e = 1$ and $m=2$ in Lemma \ref{lem:height}.  
Our task is then to show that $n_i(B)\ll B^{1+\ve}$ for $i=1,2$.
Without loss of generality we shall show this for $i=1$, with 
$$
n_1(B)=\#\left\{x \in U(k):  H_k(x) \leq B \mbox{ and } H_k(f_1(x)) \leq c_1 B^{1/2}\right\}
$$ 
and an appropriate constant $c_1=c_1(X)>0$.

For each $p \in \mathbb{P}^1(k)$ with  $H_k(p)\leq c_1B^{1/2}$ we can choose a representative $(u,v) \in Z_2$, by Lemma \ref{important1}.
Let $n_1 (u,v, B)$ be the number of points in $f_1^{-1}([u,v]) \cap U(k)$ with height at most $B$.  
Then 
$$
n_1 (B) \leq \sum_{\substack{(u,v ) \in Z_2  \\  H_k([u,v]) \leq c_1 B^{1/2}} } n_1(B;u,v).$$ 
Given $\A\in (\RR_{\geq 1})^{s_k}$, 
we split the right hand side into dyadic intervals, writing 
\beq\label{n1BS} 
n_1 (\A,B) = 
\sum_{\substack{(u,v ) \in Z_2  \\  
A_\nu\leq \sup\{\|u\|_\nu, \|v\|_\nu\} <2A_\nu
}}
n_1(B;u,v).\eeq 
Let  $A=\|\A\|$. 
It then follows from  \eqref{Hasymp} that any point $(u,v)$ in the sum satisfies 
$A\ll H_{k}([u,v])\ll A$.
We are clearly only interested in  $A \ll B^{1/2}$.  

Now $(ux,yv,xv,yu,z) \in f_1^{-1}([u,v])$ if and only if
$$
Q(ux,yv,xv,yu) + a z^2 = 0.
$$
This is a conic bundle torsor $\mathcal{T}$, as in \eqref{bundleshape2}, with  $\deg \Delta(u,v) = 4$.  
On multiplying  $(x,y,z)$ by an appropriate scalar,  Lemma \ref{important1} ensures that
we will have $\langle ux,yv,xv,yu,z\rangle = \mathfrak{a}_i$ for some $i\in \{1,\dots,h\}$ 
and 
\beq\label{eq421} 
\| (ux,yv,xv,yu,z) \|_\star \leq c_2 H_k ([ux,yv,xv,yu,z])^{1/s_k},
\eeq 
for some constant $c_2>0$.  
We must count the number of such points which lie on  $\mathcal{T}$. Moreover, it suffices to work on 
the open set $\mathcal{T}_0$ since we wish to avoid points lying on lines in $X$.

Our goal is to  apply Theorem \ref{conicbundleest}. A triple $(x,y,z)$ satisfying the above restrictions does not necessarily have  $x,y\in  \mathfrak{o}$.  
On multiplying  $(x,y,z)$ by a suitable scalar and adjusting  $c_2$ appropriately
in \eqref{eq421}, however, we can proceed
 under the assumption that $(x,y,z) \in Z'_3$.
In conclusion,  $n_1(B;u,v)$ is bounded   by the number of elements $(x,y,z) \in Z'_3$ satisfying $
(ux,yv,xv,yu,z) \in U(k)$ and 
\beq\label{redefinen1} 
\sup \left\{\|ux\|_\nu,  \|yv\|_\nu ,\|xv\|_\nu, \|yu\|_\nu,\|z\|_\nu\right\} \leq c_2 B^{1/s_k}.
\eeq  
We redefine $n_1(B;u,v)$ to be this cardinality.  

By \eqref{n1BS} and \eqref{redefinen1}, for every point counted by $n_1(B;u,v)$, there is an absolute constant $c_3>0$ such that
$$
\|x \|_\nu , \|y \|_\nu \leq \frac{c_2 B^{1/ s_k}}{ \sup\{ \| u \|_\nu, \| v \|_\nu  \}} \leq c_3B^{1/{s_k}}/A_\nu,
$$ 
for every $\nu\mid \infty$. Moreover, we have  $\| z \|_\nu \leq c_2 B^{1/ s_k} $, for each $\nu\mid \infty$.  
Hence we can apply Theorem \ref{conicbundleest} with
$r_{1,\nu} = r_{2,\nu} = c_3B^{1/{s_k}}/A_\nu$ and  $r_{3, \nu} = c_2 B^{1/ s_k}$ and $n = 4$.  This yields the estimate
$$
n_1 (\A,B) \ll  A^{2+\eps} +  B A^{\eps}.
$$ 
Summing over dyadic values of $A_\nu$, with  $A=\|\A\| \ll B^{1/2}$, therefore leads to the desired bound $n_1 (B) \ll B^{1 + \eps}$.

\subsection{Proof of Theorem \ref{dp3thm}}  

The argument in this section and the next is similar to the proof of Theorem \ref{dp4thm} and so we shall allow ourselves to be more concise.
Suppose $X$ is a del Pezzo surface of degree 3 over $k$, with  three coplanar lines defined over $k$.  
After a possible change of variables we may assume that $X\subset \PP^3$ is given by 
$$
L_1 L_2 L_3 = x_0 Q,
$$ 
where each $L_i\in \mathfrak{o}[x_1,x_2,x_3]$ is a linear form and  $Q\in \mathfrak{o}[x_0,\dots,x_3]$ is a quadratic form. Following Broberg \cite{Broberg}, we define three conic bundle morphisms $f_1, f_2, f_3: X\rightarrow \PP^1$ via
\begin{align*}
f_1(x) &=  
\begin{cases}
[x_0 , L_1] &\textrm{ if } (x_0 ,L_1) \neq (0,0),  \\
[L_2 L_3 , Q] & \textrm{ if } (L_2 L_3 , Q) \neq (0,0), \end{cases}\\
f_2(x) &= \begin{cases}
[x_0 , L_2] &\textrm{ if } (x_0 ,L_2) \neq (0,0),  \\
[L_1 L_3 , Q] & \textrm{ if } (L_1 L_3 , Q) \neq (0,0), 
\end{cases}\\
f_3(x) &= \begin{cases}
[x_0 , L_3] &\textrm{ if } (x_0 ,L_3) \neq (0,0),  \\
[L_1 L_2 , Q] & \textrm{ if } (L_1 L_2 , Q) \neq (0,0). 
\end{cases} \end{align*}
These morphisms are all well-defined, since $X$ is non-singular.
In the  notation of Section \ref{heights} we have  
$$
\psi \circ f (x) = [x_0^2, x_0 L_3, x_0 L_2 , x_0 L_1, L_2 L_3 , L_1 L_3 , L_1 L_2 , Q ],
$$ 
for all $x\in U$,
so we take $e = 2$ and $m=3$ in Lemma \ref{lem:height}.
We need to show that 
$n_i(B) \ll B^{4/3+\eps}$, for $1\leq i\leq 3$. Without loss of generality we  shall do so for $i=1$, with 
$$
n_1(B)=\#\left\{x \in U(k):  H_k(x) \leq B \mbox{ and } H_k(f_1(x)) \leq c_1 B^{2/3}\right\}.
$$

After a change of variables we may assume that $L_1 = x_1$ and $L_2 = x_2$.  
We look at the fibres of $f_1 : X \rightarrow \mathbb{P}^1$.  The fibre  above a point $[u,v]$ is the set of points 
$[u y_1 , v y_1 , y_2 , y_3]$ where $y_1 \neq 0$ and 
\begin{equation}\label{eq:nose}
v y_2 L_3(v y_1,y_2 , y_3) = u Q(u y_1 , v y_1, y_2, y_3) .
\end{equation}
This is a conic bundle torsor $\mathcal{T}$, as in \eqref{bundleshape2}, with 
$\deg \Delta(u,v)=5$.
Let $n_1 (B;u,v)$ be the number of points in $f_1^{-1}([u,v]) \cap U(k)$ with height at most~$B$.  Then 
$$
n_1 (B) \leq \sum_{\substack{(u,v ) \in Z_2  \\  H_k([u,v]) \leq c_1 B^{2/3} }} n_1(B;u,v).
$$ 
As in \eqref{n1BS},
we split the right hand side into dyadic intervals $ n_1 (\A,B) $  for suitable $\A\in (\RR_{\geq 1})^{s_k}$ such that $A=\|\A\| \ll B^{2/3}$.

Using an identical argument to that leading up to \eqref{redefinen1}, we can redefine $n_1(B;u,v)$ to be the number of points $(y_1 , y_2 , y_3) \in Z'_3$ such that 
\eqref{eq:nose} holds and  
$$ 
\sup \{ \|u y_1 \|_\nu, \|v y_1 \|_\nu, \| y_2 \|_\nu, \| y_3 \|_\nu  \} \leq c_2 B^{1/s_k}
$$ 
for some constant $c_2>0$.  
Similarly to before, we see there is a constant 
 $c_3>0$ such that 
 Theorem \ref{conicbundleest} can be applied with $r_{1,\nu} = c_3 B^{1/ s_k}/ A_\nu$ and 
$r_{2, \nu} = r_{3, \nu} = c_2 B^{1/ s_k}$ and $n = 5$.  
This shows that 
$$
n_1 (\A , B ) \ll A^{2+\eps} + B A^{\eps}, 
$$ 
and we obtain the desired conclusion by summing over dyadic values of  $A_\nu$, with 
$A \ll B^{2/3}$.

\subsection{Proof of Theorem \ref{dp2thm}}

We suppose
 that $X\subset \PP(2,1,1,1)$ is a del Pezzo surface of degree $2$, as in 
\eqref{eq:dp2}, which 
contains a non-singular conic $C$  defined over $k$.  
Following an argument suggested to us by Per Salberger, 
we will show  that it may be  given by an  equation of the form
$$
t^2 = q_1 q_2 + q_3^2,
$$ 
where $q_1,q_2,q_3\in k[x_1,x_2,x_3]$ are quadratic forms such that 
$q_1 q_2 + q_3^2$ is a non-singular quartic form. 

Any non-singular conic $C$ contained in $X$
 is an 
 irreducible curve satisfying  $(C,-K_X)=2$. Hence $C$
is mapped isomorphically onto a conic in $\PP^2$,  under the
double cover map  
$ [t,x_1,x_2,x_3]\mapsto [x_1,x_2,x_3].
$
Let us  suppose  that this conic in $\PP^2$ is  given by the equation 
$q_1=0$, for a non-singular ternary quadratic form
$q_1$ defined over $k$.
We must have one more relation between $t$ and the six quadratic monomials 
in $x_1,x_2,x_3$. This gives a further equation $t-q_3=0$ on $C$, for a quadratic form
$q_3$ defined over $k$.  Substituting this into the equation for $X$ we see that $f-q_3^2$ vanishes on the conic $q_1=0$ in $\PP^2$.
Hence $f=q_1q_2+q_3^2$ for a further 
quadratic form
$q_2$ defined over $k$, such that $q_1q_2+q_3^2$ is non-singular, 
which thereby establishes the claim.
%
%

We may henceforth  assume that $q_1,q_2,q_3$ are all defined over $\mathfrak{o}$ on absorbing a suitable constant 
into $t$. 
The 56 exceptional curves are the preimages of the 28 bitangents to the quartic plane
curve  $q_1 q_2 + q_3^2=0$. We let $U\subset X$ be the open subset which avoids all of these. 
As indicated previously, we take our height function $H_k : \mathbb{P}(2,1,1,1)(k) \mapsto \mathbb{R}_{\geq 1}$ to be
$$
[t,\xx] \mapsto \prod_{\nu\in \Omega} \sup \{   \| x_1 \|_\nu, \| x_2 \|_\nu  , \| x_3 \|_\nu \}.
$$ 
We define the morphisms $f_1,f_2:X\rightarrow \PP^1$ via
\begin{align*}
f_1 ([t,\xx]) &= 
\begin{cases}
[t-q_3, q_1]  &\textrm{if } (t-q_3, q_1) \neq (0,0),  \\
[ q_2 , t+q_3]  &\textrm{if } ( q_2 , t+q_3) \neq (0,0), 
\end{cases}\\
f_2 ([t,\xx]) &= 
\begin{cases}
[t-q_3, q_2] &\textrm{if }(t-q_3, q_2) \neq (0,0),  \\
[q_1 , t+q_3]  &\textrm{if }  (q_1 , t+q_3) \neq (0,0). 
\end{cases}\end{align*}
These morphisms are well-defined 
since $q_1q_2+q_3^2$ is non-singular.
In this setting we have  
$$
\psi \circ f ([t,\xx]) = [t- q_3, q_2, q_1, t+ q_3],
$$ 
for all $[t,\xx]\in U$,
so we take $e = 2$ and 
$m=2$ in Lemma \ref{lem:height}.
We wish to show that 
$n_i(B) \ll B^{2+\eps}$ for $i=1,2$.  Without loss of generality we shall show this for $n_1(B)$, with
$$
n_1(B)=\#\left\{x \in U(k):  H_k(x) \leq B \mbox{ and } H_k(f_1(x)) \leq c_1 B\right\}.
$$

We look at the fibres $f_1 : X \rightarrow \mathbb{P}^1$ in $U$.  Defining $n_1 (B;u,v)$ to be the number of rational points in $f_1^{-1}([u,v]) \cap U$ with height at most $B$,  we have 
$$
n_1(B) \leq \sum_{\substack{(u,v) \in Z_2 \\ H_k([u,v]) \leq c_1 B}} n_1(B;u,v).
$$
We shall consider the contribution  $ n_1 (\A,B)$ 
from dyadic intervals, 
as in \eqref{n1BS}, for $\A\in (\RR_{\geq 1})^{s_k}$ such that $A=\|\A\| \ll B$.

Suppose $[t,\xx] \in f_1^{-1}([u,v]) \cap U$ for $uv \neq 0$.  Then the point $(u,v;[\xx])$ 
is constrained to lie on the variety $\mathcal{T}\subset \AA^2\times\PP^2$, given by the equation
$$
q_1(\xx) u^2 + 2 q_3(\xx) uv - q_2(\xx) v^2 =0.
$$
This is a conic bundle torsor of  the form \eqref{bundleshape2}, with $\deg \Delta(u,v)=6$. 
Thus  
 Theorem \ref{conicbundleest} can be applied directly with $r_{i,\nu} = B^{1/s_k}$ for $i=1,2,3$ and 
 $n =6$, giving 
$$
n_1(A,B) \ll A^{2+\eps}+ B A^\eps.
$$  
Summing for dyadic  $A \ll B$ shows that $n_1(B) \ll B^{2+ \eps}$, as claimed.

\section{Ternary forms}\label{ternary}

In this section we establish Theorem \ref{qmain}.
The structure of the proof of Theorem \ref{qmain} is similar to \cite[Thm.~6]{Broberg} (which in turn follows the proof of  \cite[Thm.~2]{HB}).  
The main idea is to cover the solutions to $Q(\xx) = 0$ by a relatively small number of lattices, each of which has  a large determinant.  This is done in Section \ref{s:fermat},
after first  recalling some basic facts about lattices over number fields in Section \ref{s:lattices}.
Then, for given  $\underline\rrr = (\rrr_1, \rrr_2, \rrr_3),$ 
we obtain in Section \ref{s:uniform} a uniform estimate for the number of points $x \in \mathbb{P}^2(k)$ with representative $\xx=(x_1, x_2 , x_3) \in Z'_3,$ such that $x_i \in L(\rrr_i)$ and $Q(\xx) = 0$. 
This is then used to deduce Theorem \ref{qmain} by rescaling the lattices appropriately. 
Throughout Section~\ref{ternary} we return to our convention that all of the implied constants are allowed to depend at most upon the number field $k$. Furthermore, all of the implied constants in this section are effective.

\subsection{Lattices}\label{s:lattices} 
We say that an $\mathfrak{o}$-module $\Lambda$ in $k^n$ is an {\em $\mathfrak{o}$-lattice in $k^n$}, if it is finitely generated and contains a basis of $k^n$ over $k$.  We can define its determinant $\det \Lambda$ to be  the index 
$[\mathfrak{o}^n:\Lambda]$
as an additive subgroup.
If $\Lambda$ is an $\mathfrak{o}$-lattice in $k^n$ and $\nu\in \Omega$ is a finite place, then $\Lambda_\nu = \Lambda \otimes_{\mathfrak{o}} \mathfrak{o}_\nu$ is a free $\mathfrak{o}_\nu$-module in $k_\nu^n$ such that $\Lambda_\nu$ contains a basis for $k_\nu^n$ over $k_\nu$, where $\mathfrak{o}_\nu$ is the ring of integers of $k_\nu$.  We say that such an $\mathfrak{o}_\nu$-module is an {\em $\mathfrak{o}_\nu$-lattice in $k_\nu^n$.}
The following results are standard (see 
\cite[Thm.~4]{Broberg} and 
  \cite[Prop.~5]{Broberg}, respectively).

\begin{lemma}\label{LocalLat}  For each finite place $\nu\in \Omega$, let $L_\nu$ be an $\mathfrak{o}_\nu$-lattice in $k_\nu^n$ such that $L_\nu = \mathfrak{o}_\nu^n$ for almost all $\nu$.  If $\Lambda = \cap_{\nu\nmid \infty} L_\nu \cap k^n$, then $\Lambda$ is the unique $\mathfrak{o}$-lattice in $k^n$ such that $\Lambda_\nu = L_\nu$ for all $\nu\in \Omega$.  
\end{lemma}  
\begin{lemma}\label{adef}  If $L \subset \Gamma$ are $\mathfrak{o}$-lattices in $k^n$, then there is an element $a \in k^{*}$ such that $\Gamma \subset aL$ and $[\Gamma:L]\ll \N_k(a)$.  \end{lemma}

We define measures for the places $\nu\in \Omega$  as follows. 
If $\nu\mid \infty$ and  $k_\nu = \mathbb{R}$, then $\d \mu_\nu$ is the ordinary Lebesgue measure.  If $\nu\mid \infty$ and  $k_\nu = \mathbb{C}$, then $\d \mu_\nu$ is the Lebesgue measure multiplied by 2.  If $\nu\nmid \infty$, then $\d \mu_\nu$ is the usual $\nu$-adic measure normalised so that $\mu_\nu(\mathfrak{o}_\nu) = \| \mathfrak{D}_\nu \|_\nu$, where $\mathfrak{D}_\nu$ is the local different of $k$ at $\nu$.

For each $\nu\mid \infty$, let $S_\nu$ be a non-empty, open, convex, symmetric, bounded subset of $k_\nu^n$. 
For an $\mathfrak{o}$-lattice $\Lambda$ in $k^n$, we shall identify $\Lambda$ with its image in $S = \prod_{\nu\mid \infty} S_\nu$, under  the diagonal embedding.  We define the {\em $i$th successive minimum of $\Lambda$ with respect to $S$} to be
$$
\lambda_i = \inf \{ \lambda \in \mathbb{R}_{>0} : \mbox{$\Lambda \cap \lambda S$ contains  $i$ linearly independent vectors}\}.
$$
The following 
result is an analogue of 
{\em Minkowski's second theorem} in the ad\`{e}les due to Bombieri and Vaaler \cite{BomVal} (see the corollary to \cite[Thm.~5]{Broberg} for the present formulation).

\begin{lemma}\label{minima} If $\lambda_1 \leq  \cdots \leq \lambda_n$ are the successive minima of $\Lambda$ with respect to $S$, then 
$$   (\lambda_1 \cdots \lambda_n)^d \prod_{\nu\mid \infty} \Vol(S_\nu) \ll_{n} [\mathfrak{o}^n:\Lambda], $$  the volume $\Vol(S_\nu)$ being taken with respect to $\d \mu_\nu$. 
\end{lemma}

\subsection{Points on Fermat curves}\label{s:fermat}

In this section we shall prove a generalisation of \cite[Lemma 4.9]{Timbook} and 
\cite[Thm.~3]{TimRainer} to number fields. In fact the proof of these results contains
an error and we shall take the opportunity to correct this here. 
Let $I_k$ be the set of integral ideals of $\mathfrak{o}$.  For each integer $t \geq1$, define the multiplicative function on integral ideals  
$
\delta_t: I_k \rightarrow \NN,
$ 
via
\beq\label{fdef}\delta_t(\mathfrak{p}^r) = r+t-1 
\eeq  
for each prime ideal $\mathfrak{p}$.  
Note that $\delta_2=\tau$ is the usual divisor function on integral ideals.

\begin{lemma}\label{BroLem4}
Consider the equation 
\beq\label{dforms}
F(\xx) = a_1 x_1^t + a_2 x_2^t + a_3 x_3^t =0 
\eeq
for $a_i \in \mathfrak{o} $ and $ t \in \mathbb{Z}_{\geq 2}$.  Let $\Delta(F)$ be the principal ideal $\langle a_1 a_2 a_3\rangle$, and let $\Delta_0(F)$ be the ideal $\langle a_1 a_2, a_2 a_3, a_3 a_1  \rangle $.
Suppose $\xx \in \mathfrak{o}^3$ is a solution of \eqref{dforms}.  Then $\xx$ lies in one of at most $J$ $\mathfrak{o}$-lattices $\Gamma_1 , \dots , \Gamma_J \subset \mathfrak{o}^3$, such that
\begin{enumerate}[(i)] \item $J \leq t^{3d} \delta_t(\Delta(F))$;
\item 
for each $j\leq J$ we have 
$\dim \Gamma_j = 3$ and 
$$\det \Gamma_j \geq \frac{t^{-2d} \N_k (\Delta(F))^{2/t}}{ \N_k(\Delta_0(F))^{{3/t}}}.$$   \end{enumerate} 
\end{lemma}

\begin{rem}
When $k=\QQ$, 
 \cite[Lemma 4.9]{Timbook} and 
\cite[Thm.~3]{TimRainer}  record a version of this result with  the  factor 
$t^{\omega(a_1a_2a_3)}$ instead
of our $\delta_t(\Delta(F))=\delta_t(a_1a_2a_3)$,
where $\omega(n)$ is the number of distinct prime divisors of an integer $n$ (note that $\delta_t(p^r)=r+t-1\geq t=t^{\omega(p^r)}$ for  any prime $p$ and any $r\in \NN$).
However, there is an error in the proof of these results which invalidates this bound. In addition to providing a generalisation to arbitrary number fields, Lemma 
\ref{BroLem4} corrects this error. Moreover, one easily shows that 
nothing has been lost on average, since 
$$
\sum_{\substack{\mathfrak{a}\subset \mathfrak{o}\\
\N_k(\mathfrak{a}) \leq B}} \delta_t(\mathfrak{a}) \ll B (\log B)^{t-1},
$$
for any $t\geq 2$.
\end{rem}

\begin{proof}[Proof of Lemma \ref{BroLem4}] Suppose $\mathfrak{p} \mid \Delta(F)$ is a prime ideal, let $\mathfrak{o}_\mathfrak{p}$ be the localisation of $\mathfrak{o}$ at $\mathfrak{p}$ and put $\mathfrak{q} = \mathfrak{p} \mathfrak{o}_\mathfrak{p}$.  Suppose that that $q = \N_k(\mathfrak{p}) = p^l $ for some rational prime $p$, so that $\mathfrak{o}_\mathfrak{p} / \mathfrak{q} \cong \mathbb{F}_q$.  Let $\nu$ be the place associated to $\mathfrak{p}$ and suppose that $\pi$ is a uniformizer of $\mathfrak{o}_\mathfrak{p}=\mathfrak{o}_\nu$.  Finally, put $\gamma = 2\,\text{ord}_\nu(t)$ and note that $\N_k(\mathfrak{p})^\gamma=q^{2\text{ord}_\nu(t)}\mid t^{2d}$, if $\mathfrak{p}^{\text{ord}_\nu(t)}\mid t$.

We suppose that 
$$
F(\xx) = \epsilon_1 \pi^{\alpha_1} x_1^t +  \epsilon_2 \pi^{\alpha_2} x_2^t + \epsilon_3 \pi^{\alpha_3}x_3^t,
$$
for units $\epsilon_i$ in $\mathfrak{o}_\mathfrak{p}$ and $\alpha_i \in \mathbb{Z}_{\geq 0}$ such that   $\alpha_1 \leq \alpha_2 \leq \alpha_3$.
Let $a_\nu$ and $b_\nu$ be the non-negative integers defined by $\| \Delta(F) \|_\nu = \| \pi \|^{a_\nu}_\nu$ and  $\| \Delta_0(F) \|_\nu = \| \pi \|^{b_\nu}_\nu$. 
Then $a_\nu = \alpha_1 + \alpha_2 + \alpha_3$ and $ b_\nu = \alpha_1 + \alpha_2 $.  Hence
\beq\label{bastard} \N_k(\mathfrak{p})^{(2\alpha_3 - \alpha_1 - \alpha_2)/t}= \N_k(\mathfrak{p})^{(2a_\nu - 3b_\nu)/t}.\eeq 
Suppose that  $\xx \in \mathfrak{o}_\mathfrak{p}^3 $, with $Q(\xx) = 0$. 
We will show that there exist $\mathfrak{o}_\nu$-lattices
$M_1, \dots , M_K \subset \mathfrak{o}_\mathfrak{p}^3$ of dimension 3, such that $\xx \in M_i $ for some $i\in \{1,\dots,K\}$, with 
$$
K \leq \begin{cases} \alpha_3 -1 +t & \mbox{if $\gamma = 0$}, \\
(\alpha_3 -1 +t)\N_k(\mathfrak{p})^{\gamma+1} & \mbox{if $\gamma > 0$}, \end{cases}
$$
and
$$
\det M_i \geq \N_k(\mathfrak{p})^{(2 \alpha_3 - \alpha_1- \alpha_2) / t -\gamma},
$$
for  $1 \leq i\leq K$.  
Using the Chinese remainder theorem, 
we may then deduce the result
by taking the product over all prime ideals such that $\mathfrak{p}\mid \Delta(F) $ and recalling \eqref{bastard}.

Suppose that $x_i = \pi^{\xi_i} u_i$ for $i=1,2$, with  $u_1,u_2$ units in $\mathfrak{o}_\mathfrak{p}$.  Then 
$$
\epsilon_1 {u}_1^t \pi^{\alpha_1 + t \xi_1}  +  \epsilon_2  {u}_2^t \pi^{\alpha_2+ t \xi_2} \equiv 0 \bmod{\mathfrak{q}^{\alpha_3}}.
$$  
We split into cases as in the proof of \cite[Lemma 4.9]{Timbook}.  The oversight in that proof was that the contributions from the different cases were not added up correctly at the end, and this turns out to be fairly delicate.
The $\xx$ in which we are interested satisfy
\begin{itemize}
\item[(I)] $\alpha_3 \leq \min_{i=1,2}\{ \alpha_i + t \xi_i \}$; or 
\item[(II)] $\alpha_3 > \max_{i=1,2}\{ \alpha_i + t \xi_i \}$. 
\end{itemize}
Note that it is impossible for $\alpha_3$ to be between the two.

Let $L_1$ be the lattice
\beq\label{L1def}
L_1 = \{ \xx \in \mathfrak{o}_\mathfrak{p}^3 : \mbox{$x_i \in \mathfrak{q}^{\max\{0,\left\lceil \frac{\alpha_3 - \alpha_i -\gamma}{t} \right\rceil\}}$  for $i=1,2$} \}.
\eeq 
The determinant of $L_1$  is at least 
$$ 
\N_k(\mathfrak{p})^{
\max\{0,\left\lceil \frac{\alpha_3 - \alpha_2 -\gamma}{t} \right\rceil\}
+\max\{0,\left\lceil \frac{\alpha_3 - \alpha_1 -\gamma}{t} \right\rceil\}}
\geq  \N_k(\mathfrak{p})^{(2 \alpha_3 - \alpha_1 - \alpha_2) / t -\gamma },
$$ 
since $t \geq 2$ . Any $\xx$ from  Case I must lie in $L_1$, since $\gamma \geq 0$.  Hence the points in Case I can be covered by one lattice of the required determinant.

For the points from Case II we have  $\alpha_1 + t \xi_1 = \alpha_2 + t \xi_2 = \eta$, say.  
Note that there are $ \left\lfloor \frac{\alpha_3 - \alpha_2-1}{t} \right\rfloor +1$ possibilities for $\eta$. 
If $\alpha_3 - \eta \leq \gamma$ then 
 it is easy to see that $\xx \in L_1$, and so we are done.

Alternatively, we suppose that  $\alpha_3 - \eta > \gamma$
and  
$$
(u_1 / u_2)^t \equiv - \epsilon_2 / \epsilon_1 \bmod{\mathfrak{q}^{\alpha_3-\eta} }.
$$
Now, $y^t \equiv \epsilon \bmod{\mathfrak{q}}$ has at most $\text{gcd}(t, q-1) \leq t$ roots, since $\mathfrak{o}_\mathfrak{p}/\mathfrak{q} \cong \mathbb{F}_q$.
Hensel's lemma tells us that the congruence 
$y^t \equiv \epsilon \bmod{ \mathfrak{q}^{\alpha_3-\eta}}$
 has the same number of solutions as the congruence $y^t \equiv \epsilon \bmod{ \mathfrak{q}^{\gamma}}$, 
since $\alpha_3 - \eta > \gamma$. The total number of solutions is therefore bounded above by $t \N_k(\mathfrak{p})^{\gamma}$. It follows that  there exist 
$r_1,  \dots , r_H \in \mathfrak{o}_\mathfrak{p} / \mathfrak{q}^{\alpha_3-\eta}$, where $H \leq  t\N_k(\mathfrak{p})^{\gamma}$, such that 
$$
u_1 \equiv r_i  u_2 \bmod{\mathfrak{q}^{\alpha_3-\eta}},
$$ 
for some $i\in \{1,\dots,H\}$.  
 Every solution $\xx \in \mathfrak{o}_\mathfrak{p}^3$ which satisfies this congruence
 lies in the lattice  defined by the conditions
 \beq\label{latform}
 x_i = \pi^{\xi_i} x'_i, \quad \quad x'_1 \equiv r_i x'_2 \bmod{\mathfrak{q}^{\alpha_3 - \eta}},
 \eeq 
 for $x'_i \in \mathfrak{o}_\mathfrak{p}$.  This has determinant 
\begin{align*}
\N_k(\mathfrak{p})^{ \alpha_3 + \xi_1 + \xi_2 - \eta} \geq \N_k(\mathfrak{p})^{ (2 \alpha_3 - \alpha_1 - \alpha_2)/t}  
\end{align*} 
in $\mathfrak{o}_\mathfrak{p}^3$, which is satisfactory.

Now we count up the total number of lattices.  First suppose that $\gamma = 0$ and $\alpha_3 - \alpha_2 \equiv 1 \bmod{t}$.
Then  for each $\xi_i$ arising in Case II we have  
\beq\label{L1slap}
\xi_i\leq \left\lfloor \frac{\alpha_3 - \alpha_i-1}{t} \right\rfloor  
=\frac{\alpha_3 - \alpha_i-1}{t},
\eeq  
since then $\alpha_1\equiv \alpha_2 \bmod{t}$.
On the boundary case, we have $\eta = \alpha_3-1$.  But then, if it arises,  this gives us a lattice of the form \eqref{latform} with the exponent of $\mathfrak{q}$ being 1.  Thus from \eqref{L1def} and \eqref{L1slap}, we see that $L_1$ is a subset of these lattices, so we need not include it in our count.  The total number of lattices is therefore found to be at most
\begin{align*} 
t \left( \left\lfloor\frac{\alpha_3 - \alpha_2-1}{t} \right\rfloor + 1 \right) &\leq t  \left(\frac{\alpha_3 - 1}{t} + 1 \right) \\
&=  \alpha_3 - 1 + t,
\end{align*}
which is satisfactory.

Next suppose that $\gamma = 0$ and $\alpha_3 - \alpha_2 \not\equiv 1 \bmod{t}$.  Either $\alpha_3 - \alpha_2=0$, in which case the second case cannot happen at all (so we need one lattice in total), or $\alpha_3 - \alpha_2 \geq 2.$  But then,
when we add $L_1$ to the count, 
the total number of lattices is at most 
\begin{align*} t  \left(\left\lfloor\frac{\alpha_3 - \alpha_2-1}{t} \right\rfloor + 1 \right) +1
&\leq   t  \left(\frac{\alpha_3 - \alpha_2 -2}{t}  + 1 \right) + 1 \\ 
&\leq \alpha_3 - 1 + t,
\end{align*}
which is also satisfactory.

Finally suppose that $\gamma > 0$.  In this case the total number of lattices is at most
\begin{align*} 
t \N_k(\mathfrak{p})^{\gamma}\left( \left\lfloor\frac{\alpha_3 - \alpha_2-1}{t} \right\rfloor + 1\right) +1 &\leq t \N_k(\mathfrak{p})^{\gamma} \left(\frac{\alpha_3 - 1}{t} + 1 \right) +1 \\
&= \N_k(\mathfrak{p})^{\gamma}( \alpha_3 - 1 + t ) +1 \\
& \leq  \N_k(\mathfrak{p})^{\gamma+1}( \alpha_3 - 1 + t ). 
\end{align*}
This too 
 is satisfactory and so completes the proof of the lemma.
\end{proof}

We now turn to the setting of  Theorem \ref{qmain}, working over  
each $\mathfrak{o}_\mathfrak{p}$ separately as in the proof of the last lemma.  
Exactly as in \cite[Lemma~4(b)]{Broberg}, after diagonalisation of the quadratic form $Q$
it suffices to analyse equations of the shape \eqref{dforms} with $t =2$. 
We obtain the following result. 

\begin{corollary}\label{BroLem12}Let $Q$, $\Delta(\mathbf{M})$, $\Delta_0(\mathbf{M})$  be as in Theorem \ref{qmain}.  Suppose $\xx \in \mathfrak{o}^3$ is a solution of $Q(\xx) =0$.  Then $\xx$ lies in one of at most $J$ $\mathfrak{o}$-lattices $\Gamma_1 , \dots , \Gamma_J \subset \mathfrak{o}^3$ such that 
\begin{enumerate}[(i)] 
\item $J \ll \tau (\Delta(\mathbf{M}))$;
\item for each $j\leq J$ we have 
$\dim \Gamma_j = 3$ and 
$$
\det \Gamma_j \gg  \frac{\N_k (\Delta(\mathbf{M})) }{
\N_k(\Delta_0(\mathbf{M}))^{{3/2}}}.$$  
\end{enumerate}  
\end{corollary}

\subsection{A uniform bound for rational points on conics}\label{s:uniform}

We now state and prove our 
generalisation of \cite[Thm.~6]{BrowningHB} to number fields.

\begin{theorem}\label{quadmain}   Let $Q$ be a non-singular ternary quadratic form 
and suppose that we are given  $ \rrr_1, \rrr_2, \rrr_3\in (\RR_{\geq 1})^{s_k}$. Let 
$R = \| \rrr_1 \|\| \rrr_2 \|\| \rrr_3 \|$
 and let 
$$
N(Q,\underline\rrr) = \# \left\{ x=[\xx] \in \mathbb{P}^2(k) : \mbox{$Q(\xx)=0$ and $x 
\in L(\underline\rrr) \cap Z'_3$} \right\}.
$$  
Then 
$
N(Q,\underline\rrr) \ll R^{1/3}.$
\end{theorem}

Adopting the notation from Section \ref{intro}, and applying 
Lemma \ref{important1},
 we obtain the following immediate
consequence.

\begin{corollary}
Let $C\subset \PP^2$ be an irreducible conic defined over a number field $k$. 
Then we have 
$
N(C,k,B)=O(B).
$
\end{corollary}

The proof of Theorem \ref{qmain} follows on combining Corollary \ref{BroLem12} with 
Theorem~\ref{quadmain} exactly as in the proof of \cite[Thm.~6]{Broberg}. The argument is essentially a repetition of the final stages of the proof of Theorem \ref{quadmain}, working instead with one of the lattices $\Gamma_j$.

\begin{proof}[Proof of Theorem \ref{quadmain}]
Our argument is a straightforward generalisation of 
 \cite[Thm.~6]{BrowningHB} to number fields.
We may suppose that 
$$
Q(\xx)=\sum_{1\leq i\leq j\leq 3} a_{ij}x_ix_j,
$$
with $(a_{11} , \dots , a_{33}) \in Z_{6}$.
Let  $\mathbf{M} \in \mathrm{GL}_3(\mathfrak{o})$
be the  underlying matrix.

We begin by choosing integral prime ideals $\mathfrak{p}_1 , \dots, \mathfrak{p}_r$, with \begin{equation}\label{pdef}c R^{1/3} \leq \N_k(\mathfrak{p}_1) < \cdots < \N_k(\mathfrak{p}_r) \ll R^{1/3}\eeq for a constant $c$ and some fixed $r$ to be chosen later.  This is possible because of the bounds of Chebyshev type on the number of prime ideals of $\mathfrak{o}$ of bounded norm. Note that this step would be an obstruction to proving a result in which the implied constant is only allowed to depend on  the degree of the number field $k$.
Now, either there exists some $i\in \{1,\dots, r\}$ such that $\mathfrak{p}_i \nmid \Delta(\mathbf{M})$, or else 
\beq\label{Delta1} 
\N_k(\Delta(\mathbf{M})) \geq \prod_{i=1}^r \N_k(\mathfrak{p}_i) \gg R^{r/3}. 
\eeq  
We shall suppose that \eqref{Delta1} holds.  

Define the height $H(Q)$ of $Q$ to be the height $H_k([a_{11}, \dots , a_{33}])$
and put $\|Q\|_{\star} =  \|(a_{11} , \dots , a_{33})\|_\star$.  We see that 
$$
\| Q \|_\star^{3 s_k} \gg
\|\det \mathbf{M}\|_\star^{s_k}
\geq
\N_k(\Delta(\mathbf{M})),
$$  
by \eqref{Normupper}, and $ H(Q)^3 \gg \| Q \|_\star^{3 s_k}$, 
by \eqref{Hasymp}.
Hence 
\beq\label{Q lower} H(Q) \gg R^{r / 9} \geq  B^{r / 9}, 
\eeq where $B = \prod_{\nu\mid \infty} \sup \{r_{1,\nu}, r_{2,\nu}, r_{3,\nu}\}$.  

Next 
note that any solution with $x  \in \mathbb{P}^2(k)$, with $x \in L(\underline\rrr) \cap Z'_3$ satisfies $H_k(x) \leq B$.  Suppose that  $Q = 0$ has at least $5$ solutions of height at most $B$ and   suppose that they have representatives $\xx^{(1)} \dots , \xx^{(5)} \in Z_3,$ such that $\|\xx^{(i)} \|_\star \ll B^{1/s_k}$, for $1\leq i\leq 5$.  
Consider the $5 \times 6$ matrix $\mathbf{C}$, whose $i$th row consists of the 6 possible monomials of degree 2 in the variables $x_1^{(i)},x_2^{(i)},  x_3^{(i)}$.  Then if the vector $\ff \in \mathfrak{o}^6$ has entries which are the corresponding coefficients of $Q$, we will have $\mathbf{C} \ff = \mathbf{0}$.  Also, since $\mathrm{rank}(\mathbf{C})\leq 5$, the equation 
$\mathbf{C} \ggs = \mathbf{0}$ has a non-zero integer solution $\ggs$ constructed out of the $5 \times 5$ subdeterminants of $\mathbf{C}$.  
Note that each element $c_{ij}$ of $\mathbf{C}$ has $\|c_{ij}\|_{\star} \ll B^{2/s_k}$, so that $\ggs$ satisfies $\| \ggs \|_\star \ll B^{10 / s_k}$.  
Let $G$ be the ternary quadratic form corresponding to the vector $\ggs$.  By \eqref{Hasymp}, we have  $H ( G ) \ll B^{10}$.  Note that $G$ and $Q$ have at least 5 common zeros, namely  $\xx^{(1)}, \dots , \xx^{(5)}$.  This contradicts B\'ezout's theorem unless $G$ is a constant multiple of $Q$, since $Q$ is non-singular.  In this case, therefore, we have  $ H(Q) =  H(G) \ll B^{10}$.
Comparing this with \eqref{Q lower}, we obtain a contradiction for large $R$  if we take $r > 90$. Thus we may conclude that $Q = 0$ has at most $4$ solutions  of height at most $B$, which is satisfactory.

\medskip

We proceed to consider the case   $\mathfrak{p}_i \nmid \Delta(\mathbf{M})$,
for some index $i\in \{1,\dots, r\}$. Thus we may  suppose that there is a prime ideal $\mathfrak{p}$
satisfying   $$c R^{1/3} \leq \N_k(\mathfrak{p}) \ll R^{1/3},$$ 
with $\mathfrak{p}\nmid \Delta(\mathbf{M})$. We shall suppose that $R$ is large enough to ensure that  $\mathfrak{p} \nmid \mathfrak{a}_i$ for any $i\in \{1,\dots,h\}$.  
Let $\mathfrak{o}_\mathfrak{p}$ be the localisation of $\mathfrak{o}$ at $\mathfrak{p}$, and put $\mathfrak{q} = \mathfrak{p} \mathfrak{o}_\mathfrak{p}$.
We have $\mathfrak{o}_\mathfrak{p} / \mathfrak{q} \cong \mathbb{F}_q$, where $q = \N_k(\mathfrak{p}) = p^l$ for some rational prime $p$.  
If we look at the image 
$\overline Q$ (over $\mathbb{F}_q$) 
of $Q \bmod{\mathfrak{p}}$, under this isomorphism, 
then $\overline Q$ is non-singular.  
The projective variety $\overline Q = 0$ has exactly $q$ points over $\mathbb{F}_q$. Our goal is to show that there are at most 2 points counted by $N(Q,\underline\rrr)$, for each of the corresponding cosets of $\mathfrak{o}_\mathfrak{p} / \mathfrak{q}$.  This will complete the proof of Theorem \ref{quadmain}, since we have assumed that $ q \ll R^{1/3} .$  

Fix a vector $\xx \in (\mathfrak{o}_\mathfrak{p} / \mathfrak{q})^3 \setminus \{\mathbf{0}\}$, with $Q(\xx) \equiv 0 \bmod{\mathfrak{q}}$ 
and 
\beq\label{nabla}
\nabla Q(\xx)\not \equiv \mathbf{0}\bmod{\mathfrak{q}}.
\eeq
We claim that there exists a vector $\xx^{(1)} \in \mathfrak{o}_\mathfrak{p}^3$,
with  $\xx^{(1)} \equiv \xx \bmod{\mathfrak{q}}$, 
which 
 satisfies
  $Q(\xx^{(1)}) \equiv 0 \bmod{ \mathfrak{q}^2}$ and \eqref{nabla}.
To see this we write  $\xx^{(1)} = \xx + \pi \yy^{(1)}$, for some uniformizer $\pi \in \mathfrak{q}$.  Then  $Q(\xx^{(1)}) \equiv 0 \bmod{\mathfrak{q}^2}$ if and only if $$ \yy^{(1)}. \nabla Q(\xx) \equiv - \pi^{-1} Q(\xx) \bmod{\mathfrak{q}},
$$ 
and this is clearly solvable for $\yy^{(1)}.$
This establishes the claim.

We shall count points $w \in \mathbb{P}^2(k)$ which have at least one representation as $\ww \in Z'_3$ satisfying $Q(\ww) = 0$ and  $w_i \in L(\rrr_i)$, and such that there exists $\lambda \in \mathfrak{o}_\mathfrak{p}$ with $\ww \equiv \lambda \xx^{(1)} \bmod{\mathfrak{q}}$.  
Then there is 
a vector $\zz \in \mathfrak{o}_\mathfrak{p}^3$ such that $\ww =  \lambda \xx^{(1)} + \pi \zz$.  It follows that
\begin{align*}
0 = Q(\ww) 
&\equiv \lambda^2 Q(\xx ^{(1)}) + \pi \lambda  \zz . \nabla Q(\xx^{(1)}) \bmod{\mathfrak{q}^2}\\ 
&\equiv \pi \lambda \zz . \nabla Q(\xx^{(1)}) \bmod{\mathfrak{q}^2}.
\end{align*}
Moreover, we note that $\lambda \not \in \mathfrak{q}$, since otherwise $\ww = \lambda \xx^{(1)} \bmod{\mathfrak{q}}$ implies that the ideal which spans the elements of $\ww$ is divisible by $\mathfrak{p}$, contradicting the fact that $\ww \in Z'_3$ and 
$\mathfrak{p} \nmid \mathfrak{a}_i$.  
Hence we conclude that $\zz . \nabla Q(\xx^{(1)}) \in \mathfrak{q}.$  It follows that
\begin{align*}  
\ww. \nabla Q(\xx^{(1)}) &= \lambda \xx^{(1)} . \nabla Q(\xx^{(1)}) + \pi \zz . \nabla Q(\xx^{(1)}) \\ 
&=  2\lambda  Q(\xx^{(1)}) + \pi \zz . \nabla Q(\xx^{(1)}) \\
& \equiv 0 \bmod{\mathfrak{q}^2}.
\end{align*} 
In conclusion, we have shown that any $\ww$ as above belongs to the set
$$
L_{\mathfrak{p}} =
\left\{ 
\ww \in \mathfrak{o}_\mathfrak{p}^3 : 
\begin{array}{l}
\mbox{$\ww \equiv \lambda \xx \bmod{\mathfrak{q}}$ for some $\lambda \in \mathfrak{o}_\mathfrak{p}$}\\
\ww . \nabla Q(\xx^{(1)}) \equiv 0 \bmod{\mathfrak{q}^2 }
\end{array}{}
\right\}.
$$
A simple generalisation of the proof of  \cite[Lemma 7]{BrowningHB} shows that 
$L_{\mathfrak{p}}$
is independent of the choice of $\xx^{(1)}$ and that it is
an $\mathfrak{o}_{\mathfrak{p}}$-lattice of dimension $3$ and determinant $\N_k(\mathfrak{p})^3.$  
We shall not give details of this argument here.

Define $L_\nu$ to be $\mathfrak{o}_\nu$ for all $\nu$ such that $\nu\nmid \infty$ and  $\nu\nmid \mathfrak{p}$.  Lemma \ref{LocalLat} implies that there is a unique $\mathfrak{o}$-lattice $\Lambda$ such that $\Lambda_\nu = L_\nu$ for all $\nu\in \Omega$, with 
$$
\det(\Lambda) = [\mathfrak{o}^3 : \Lambda] = \prod_{\nu\nmid \infty}[\mathfrak{o}^3_\nu : L_\nu] = [\mathfrak{o}^3_\mathfrak{p} : L_\mathfrak{p}]=\N_k(\mathfrak{p})^3.
$$
For $\nu\mid \infty$, 
consider the sets 
$$
S_\nu = \{(x_1, x_2, x_3) \in k_\nu^3 : \mbox{$|x_i|_\nu \leq r_{i,\nu}^{1/{d_\nu}}$ for $i=1,2,3$}  \},
$$ 
and put $S = \prod_{\nu\mid \infty} S_\nu$.  
We have $\mathfrak{o}^3\cap S=L(\rrr_1)\times L(\rrr_2)\times L(\rrr_3)$.
Moreover, 
$S$ is symmetric and 
$\Vol(S) \gg R$.  

Next we consider the successive minima $\lambda_1 \leq \lambda_2 \leq \lambda_3$ of $\Lambda$ with respect to $S$. By Lemma \ref{minima}, we know that 
$$(\lambda_1 \lambda_2 \lambda_3)^d \Vol(S) \ll 
[\mathfrak{o}^3:\Lambda].
$$  
It follows that 
$$
(\lambda_1 \lambda_2)^d \ll \frac{\N_k(\mathfrak{p})^2}{R^{2/3}}.
$$
It is evident from the definitions that we can find linearly independent vectors $\uu_1, \uu_2, \uu_3$  such that $\uu_i \in \Lambda\cap \lambda_i S$.
If $u_{ij}$ is the $j$th component of $\uu_i$, then  $ \|u_{ij} \|_\nu \leq \lambda_i^{d_\nu} r_{j,\nu}$ for $\nu\mid \infty$.  
Hence, if $\ww = y_1 \uu_1 + y_2 \uu_2 + y_3 \uu_3 \in S$ for some $(y_1, y_2, y_3) \in k^3$, and $\mathbf{U}$ is the matrix with columns $\uu_1, \uu_2 , \uu_3$, then for each $\nu\mid \infty$ 
we have 
\begin{equation*}  \| y_3 \|_\nu = \frac{1}{\| \det \mathbf{U}\|_\nu} \left\|  \det 
\begin{pmatrix}
u_{11} &u_{21} &w_1  \\
u_{12} &u_{22} &w_2 \\
u_{13} &u_{23} &w_3 \end{pmatrix}  \right\|_\nu  \ll \frac{r_{1,\nu}r_{2,\nu}r_{3,\nu} (\lambda_1 \lambda_2)^{d_\nu}}{\| \det \mathbf{U} \|_\nu},   
\end{equation*}  
by Cramer's rule.   
We note that $\{\uu_1, \uu_2 , \uu_3\}$ is not necessarily a basis for $\Lambda$ over $\mathfrak{o}$.  However, if we let $L$ be the free $\mathfrak{o}$-lattice with generators  $\uu_1, \uu_2 , \uu_3$, then we have that $L \subset \Lambda \subset a L$ for some $a \in k^{\times}$ such that $\N_k(a) \gg [\Lambda:L]$, by Lemma~\ref{adef}.  Hence any element $\ww \in \Lambda \cap S$ may be written as $$y_1 (a \uu_1)+  y_2 (a \uu_2)+ y_3 (a \uu_3)$$ for some $(y_1 , y_2 , y_3) \in \mathfrak{o}^3$.  

Let $Q'$ be the quadratic form given by the matrix $\mathbf{U}^T \mathbf{M} \mathbf{U}$.   Then we have shown that 
every point $w\in \PP^2(k)$ which has at least one representation $\ww\in Z_3'$
satisfying  $Q(\ww)=0$ and  $w_i \in L(\rrr_i)$,
and such that there exists $\lambda\in \mathfrak{o}_\mathfrak{p}$ with $\ww\equiv \lambda \xx^{(1)} \bmod{\mathfrak{q}}$, 
gives us a solution $(y_1 , y_2 , y_3) \in \mathfrak{o}^3$ to $Q' = 0$, with  
$$
\| y_3 \|_\nu  \ll \frac{r_{1,\nu}r_{2,\nu}r_{3,\nu} 
(\lambda_1 \lambda_2)^{d_\nu}}{\| a\|_\nu\| \det \mathbf{U}\|_\nu}.   
$$
Taking the product over all $\nu\mid \infty$ we see that 
\begin{align*}  
\N_k(y_3)  &\ll  \frac{R(\lambda_1 \lambda_2)^d}{\N_k(a)\N_k(\det \mathbf{U})} \\ 
&\ll \frac{R (\lambda_1 \lambda_2)^d }{[\Lambda:L] [\mathfrak{o}^3:L]}  \\
&=  \frac{R (\lambda_1 \lambda_2)^d }{[\mathfrak{o}^3:\Lambda]} \\
&\ll \frac{R^{1/3}}{\N_k(\mathfrak{p})}.  \end{align*}
Hence, on  taking $c$ in \eqref{pdef} sufficiently large, we deduce that $y_3=0$, whence $\ww$ is confined to the two dimensional space spanned by $\uu_1$ and $\uu_2$.  This means that the point $w \in \mathbb{P}^2(k)$ must not only lie on the irreducible conic $Q=0$, but also on a line.  There are at most $2$ such points, which thereby completes the proof of Theorem  \ref{quadmain}.
\end{proof}

\section{Sums involving binary forms}\label{binarysum}

In this section we shall prove Theorem \ref{sumestimate}.  Fix $\eps >0$.  Suppose that $F(u,v) = \beta (u + \alpha_1 v) \cdots (u + \alpha_n v)$, for $\alpha_i \in \overline{k}$ (if $F(u,v)$ has $uv$ as a factor, we can do a simple of change of variables to reach this form).  $F$ is separable by hypothesis and so  $\alpha_1,\ldots,\alpha_n$ are distinct.  We may assume that  $\beta = 1$, since the implied constant in \eqref{fc} can vary with $F$.  Set $K = k (\alpha_1 , \dots , \alpha_n)$.  For each infinite place $\nu$ of $k$, we fix an extension of $\nu$ to $K$, and extend $\| \cdot \|_\nu$ likewise.  
We shall let $K_\nu$ denote the completion at this place. 
Note that for any $(u,v) \in Z_2$ we have
$$
\N_k(F(u,v))  = \prod_{\nu\mid \infty} \|F(u,v) \|_\nu = \prod_{\nu\mid \infty} \prod_{1\leq i\leq n} \| u + \alpha_i v \|_\nu.
$$  
Let  
 $\A\in (\RR_{\geq 1})^{s_k}$ and recall the notation $\|\A\|=\prod_{\nu\mid \infty} A_\nu$.
We will show that 
\begin{equation}\label{sumstatement} \sum_{\substack{ (u,v) \in \mathfrak{o}^2 \\ A_\nu \leq \sup \{\|u\|_\nu,\|v\|_\nu\} < 2A_\nu\\  F(u,v) \neq 0}}  
\hspace{-0.2cm}
\left(\prod_{\nu\mid \infty} \prod_{1\leq i\leq n} \| u + \alpha_i v \|_\nu \right)^{-1/3} 
\hspace{-0.5cm}
\ll \|\A\|^{2 - n/3 + \eps}.
\end{equation}
Here, as throughout this section, we shall allow all implied constants to be ineffective, and to depend on $k, F$ and on the choice of $\ve$.
This will clearly suffice for the statement of Theorem \ref{sumestimate} on summing over dyadic values of $A_\nu$ such that $
A\ll \|\A\|\ll A.
$

Let $c_\nu=c_{\nu}(k,F)\geq 1$ be fixed absolute constants. 
On multiplying $(u,v)$ through by a suitable scalar, it clearly suffices to 
assume that $A_\nu\geq c_\nu$ for each $\nu\mid \infty$ when trying to prove \eqref{sumstatement}.
 Let $\mathcal{A} = \mathcal{A}(\A)$ denote the set  of $ (u,v) \in \mathfrak{o}^2$ such that \beq\label{xyrestrict} 
  A_\nu \leq \sup \{\|u\|_\nu,\|v\|_\nu\} < 2A_\nu, \quad \mbox{for all $\nu\mid \infty$,}
\eeq
and  $F(u,v)\neq 0$.
 It follows from 
 \cite[Prop.~1]{Broberg} that $$\#\mathcal{A}\leq (\#L(2\A))^2\ll \|\A\|^2.$$
 Let $(u,v)\in \mathcal{A}$.
 Since  $\alpha_1,\dots,\alpha_n$ are fixed once and for all, this implies there is a constant $C>0$ such that $\| u + \alpha_i v \|_\nu  < C A_\nu$ for all indices $i\in \{1,\dots,n\}$ and all $\nu\mid \infty$.

Let 
$$\theta = \frac{\eps}{ n }.
$$  
For any $\nu\mid \infty$ and any $(i,q) \in \{1, \dots , n\} \times \mathbb{Z}_{\geq 0}$,  
we define the sets 
$$
\mathcal{A}_\nu(i , q) = \{(u,v) \in \mathcal{A} : C A_\nu^{1- (q+1) \theta} \leq
 \| u + \alpha_i v \|_\nu  < C A_\nu^{1 - q \theta}\}.
$$
The larger that $q$ is, the more 
the factor $ \| u + \alpha_i v \|_\nu$
will contribute to the sum \eqref{sumstatement}.  The idea of the proof is that the bulk of $\mathcal{A}$ is covered by intersections of sets of the form $\mathcal{A}_\nu(i , q)$, with $q$ not too large, and we can quantify the contribution from these points very easily.  In order to handle the contribution from a set  $\mathcal{A}_\nu(i , q)$,  with  $q$  large, 
we use the fact that points in such a set  produce  good Diophantine approximations to $\alpha_i$.  Appealing to a  number field version of the Thue--Siegel--Roth theorem due to Lang, we can then show that the problem sets cannot contribute too much.

We begin with the following technical lemmas.

\begin{lemma}\label{boxlemma}  
Let $\mathbf{B}\in (\mathbb{R}_{>0})^{s_k}$. Let
$t_\nu\in K_\nu$ for each $\nu\mid\infty$ and let
$$
\mathcal{S} = \{ u \in \mathfrak{o} : \mbox{$ \|u-t_\nu \|_\nu < B_\nu$ for all $\nu\mid \infty$} \}. 
$$
Then $\#\mathcal{S} \ll 1 + \|\mathbf{B}\| $.   
The implied constant doesn't  depend on any $t_\nu$.
\end{lemma}

\begin{proof}  We may clearly assume that $\mathcal{S} \neq \emptyset$. Let  $x' \in \mathcal{S}$.  Then any $x \in \mathcal{S}$ takes the form $x = x' + \gamma$, 
with $\gamma$ belonging to the set
$$
\mathcal{S}' = \{ \gamma \in \mathfrak{o} : \mbox{$ |\gamma |_\nu < 2B_\nu^{1/d_\nu}$ for all $\nu\mid \infty$} \}.
$$ 
But this is  $L(\rrr),$ with  $r_\nu = 2 B^{1/d_\nu}_\nu$. Hence it follows from \cite[Prop.~1]{Broberg} that $\#\mathcal{S}'\ll  1 + \|\mathbf{B}\|$.   
\end{proof}

\begin{lemma}\label{countij}  
Let  $(i(\nu), q(\nu)) \in \{1, \dots , n\} \times \mathbb{Z}_{\geq 0}$ for each  $\nu\mid \infty$. Then 
\beq\label{setsizes} 
\#\left(\bigcap_{\nu\mid \infty} \mathcal{A}_\nu(i(\nu),q(\nu))\right) \ll 
\|\A\|\left\{1+\prod_{\nu\mid \infty} A_\nu^{1-\theta q(\nu)}\right\}.
\eeq 
\end{lemma}

\begin{proof} If $(u,v) \in \mathcal{A}$, then  \eqref{xyrestrict} implies that we must have  
$\|v \|_\nu \ll A_\nu$ for all $\nu\mid \infty$.  Thus, by Lemma \ref{boxlemma}, we have at most $O(\|\A\|)$ choices for $v$.  If we fix some $v$ then we must count the number of solutions $u \in \mathfrak{o}$ to the inequalities $\|u-t_\nu \|_\nu < CA_\nu^{1 - q(\nu) \theta}$ for $t_\nu = \alpha_{i(v)} v \in K$.  Now apply Lemma \ref{boxlemma} again.
\end{proof}

The next lemma makes precise the statement that if we fix a place $\nu\mid \infty$, then for any $(u,v) \in \mathcal{A}$ there is at most one $\alpha_i$ such that $\|u - \alpha_i v \|_\nu $ is ``small''. 

\begin{lemma}\label{trianglelemma} 
Let  $\nu\mid \infty$. 
There exists a constant $c(k,F)>0$ such that 
for any integers $i_1 \neq i_2$ and $q_1,q_2 \geq 1$, we have 
$$\mathcal{A}_\nu( i_1, q_1) \cap \mathcal{A}_\nu( i_2, q_2)= \emptyset, 
$$ 
if $A_\nu>c(k,F)$.		
\end{lemma} 

\begin{proof}This is a simple consequence of the triangle inequality. Suppose for a contradiction that $(u,v) \in \mathcal{A}_\nu( i_1, q_1) \cap \mathcal{A}_\nu( i_2, q_2)$, where $i_1 \neq i_2$ and $q_1,q_2 \geq 1$.  
We see that 
\begin{align*} \|\alpha_{i_1} - \alpha_{i_2} \|_\nu \|u \|_\nu &\leq
\|\alpha_{i_1} u + \alpha_{i_1} \alpha_{i_2} v \|_\nu + \|\alpha_{i_2} u + \alpha_{i_1} \alpha_{i_2} v \|_\nu  \\
&< C( \|\alpha_{i_1} \|_\nu + \|\alpha_{i_2} \|_\nu)A_\nu^{1- \theta}.
\end{align*} 
This implies that 
$\|u \|_\nu \ll A_\nu^{1- \theta}$, if $A_\nu$ is sufficiently large.  Similarly, we have  $\|v \|_\nu \ll A_\nu^{1- \theta}$.  But then 
 there is a constant $c(k,F)>0$ such that 
 this violates the restriction \eqref{xyrestrict} if 
  $A_\nu>c(k,F)$.
\end{proof}

As remarked in the paragraph following \eqref{sumstatement}, we may proceed under the assumption that each $A_\nu$ exceeds $c(k,F)$, so that Lemma 
 \ref{trianglelemma} applies. In particular, if we fix $(u,v)\in \mathcal{A}$ and a place $\nu\mid \infty$, then 
there will be at most one pair $(i_\nu, q_\nu)$ with $q_\nu \geq 1$ and $(u,v) \in \mathcal{A}_\nu(i_\nu , q_\nu)$.  
If there is no such pair, we can put $(i_\nu, q_\nu)=(1, 0)$, by default.
Thus there is a well-defined map from elements  of $ \mathcal{A}$ to $\mathcal{I} = (\{1, \dots , n\} \times \mathbb{Z}_{\geq 0})^{s_k}$.

Now we break the sum \eqref{sumstatement} into sums over those $(u,v)$ which are mapped to a particular  element
$$
\varpi= \prod_{\nu\mid \infty} (i_\nu, q_\nu) \in \mathcal{I}.
$$  
Note that there are finitely many such $\varpi$, the number depending only on $\eps$, $n$ and $k$.
Any $(u,v)$ which maps to $\varpi$ must be contained in the intersection 
$$
\mathcal{B}=\bigcap_{\nu\mid \infty} \mathcal{A}_\nu(i_\nu,q_\nu).
$$  
We begin by  considering the case in which  $\varpi$ is 
such that 
$$
1\leq 
\prod_{\nu\mid \infty} A_\nu^{1-\theta q_\nu}.
$$
Then,  when we estimate the cardinality of $\mathcal{B}$,
the second term on the right hand side of \eqref{setsizes} dominates.    
Lemma \ref{trianglelemma} implies that  if $(u,v) \in \mathcal{A}_\nu(i_\nu,q_\nu)$, then $(u,v) \in \mathcal{A}_\nu(i,0)$ for every $i \neq i_\nu$.   

Using  Lemma \ref{countij} we conclude that the contribution from the elements mapping to $\varpi$ is at most
\begin{align*} \allowdisplaybreaks 
 \sum_{(u,v)\in \mathcal{B}}  &\prod_{\nu\mid \infty} \Bigg( \| u + \alpha_{i_\nu} v \|_\nu^{-1/3}  
\prod_{i\neq i_\nu} \| u + \alpha_i v \|_\nu^{- 1/3} \Bigg) \\
&\ll \#\mathcal{B}
  \prod_{\nu\mid \infty} \Bigg( A_\nu^{-(1 - (q_\nu+1) \theta)/3}  
A_\nu^{-(n-1)(1 - \theta)/3)} \Bigg)\\ 
&\ll \|\A\|^{2-n/3}\left(
\prod_{\nu\mid \infty} 
A_\nu^{\theta\left\{-q_\nu+(q_\nu+1)/3+(n-1)/3\right\}}
\right).
\end{align*}
But the exponent of $A_\nu$ in the last line is 
$$
\theta\left\{-q_\nu+\frac{q_\nu+1}{3}+\frac{n-1}{3}\right\}\leq \frac{\theta n}{3}=\frac{\ve}{3},
$$ 
by the definition of $\theta$.
Thus  the last  line is $\ll \|\A\|^{2 - n/3 + \eps}$ and each  set $\mathcal{B}$ contributes a satisfactory amount.

\medskip

Now we must consider the set $\mathcal{C}$ of $(u,v) \in \mathcal{A}$ which map to $\varpi \in \mathcal{I}$ with 
$$
1>
\prod_{\nu\mid \infty} A_\nu^{1-\theta q_\nu}.
$$
The cardinality of this set is estimated in the following result.  

\begin{lemma}\label{lem:5.4}  We have $\#\mathcal{C} \ll \|\A\|$.
\end{lemma}

\begin{proof} Consider sets $\mathcal{C}_\nu(i,q)$ given by
$$
\mathcal{C}_\nu(i,q) = \{(u,v) \in \mathcal{A} : \| u + \alpha_i v \|_\nu  < C A_\nu^{1-q \theta}\}.
$$  
We can cover $\mathcal{C}$ with finitely many sets of the form $\bigcap_{\nu\mid \infty} \mathcal{C}_\nu(i_\nu,q_\nu)$.  
The proof of Lemma \ref{countij} then shows that the size of each set is $O(\|\A\|)$.
\end{proof}

Given an element $\xi \in k$, we define its height to be $$H_k(\xi) = \prod_{\nu\in \Omega} \sup \{1 , \| \xi \|_\nu\}.$$ 
With this in mind, 
to estimate the size of the sum over $\mathcal{C}$, we shall use the following 
generalisation of the Thue--Siegel--Roth theorem due to Lang \cite[\S 7, Thm.~1.1]{Lang}.  

\begin{lemma}[Lang's generalisation of Thue--Siegel--Roth]
\label{TSR} For each $\nu\mid \infty$, let $\alpha_\nu$ be an algebraic number over $k$ and assume that $\nu$ is extended to $k (\alpha_\nu)$.  Let $\eps' >0$.  Then the elements $\xi \in k$ satisfying the approximation condition  $$ \prod_{\nu\mid \infty} \inf \{1 , \|\alpha_\nu - \xi \|_\nu \} \leq \frac{1}{H_k(\xi)^{2 + \eps'}} $$ have bounded height.  
\end{lemma}

Suppose that $\xi = u/v$ with $(u,v) \in \mathfrak{o}^2$.  The product formula \eqref{productrule} implies that 
\begin{align*} H_k( \xi ) &= \prod_{\nu\in \Omega} \sup \{1 ,  \left\|  u/v \right\|_\nu\} \\
&=  \prod_{\nu\in\Omega} \sup\{\left\| u \right\|_\nu , \left\| v \right\|_\nu \} \\
&= H_k([u,v]). 
\end{align*}
Thus Lemma \ref{TSR} tells us that if $(u,v) \in Z_2$ and $H_k([u,v])\gg 1$, 
then 
\begin{equation}\label{Langlemma} 
\prod_{\nu\mid \infty} \inf \left\{1 ,  \left\|\alpha_\nu - \frac{u}{v} \right\|_\nu  \right\} > \frac{1}{H_k([u,v])^{2 + \eps'}}.
\end{equation}

Let $(u,v) \in \mathcal{C}$ and let $\nu\mid \infty$.  By the argument of Lemma \ref{trianglelemma}, assuming  $A_\nu$ is sufficiently large,   there can be at most one $i = i_\nu$ such that $$ \| u - \alpha_i v \|_\nu < A_\nu^{1-\eps'} .$$ Moreover, we have 
 $\N_k (v) \geq 1$ and
\begin{align*} 
\N_k(F(u,v)) &= \prod_{\nu\mid \infty} \prod_{1\leq i\leq n} 
\| u - \alpha_i v \|_\nu \\
& \geq  \prod_{\nu\mid \infty} 
(A_\nu^{1-\eps'})^{n-1}
\| u - \alpha_{i_\nu} v \|_\nu  \\
&= \|\A\|^{(n-1)(1-\eps')} \N_k(v)\prod_{\nu\mid \infty} \left\| \alpha_{i_\nu} - \frac{u}{v} \right\|_\nu\\
& \geq  \|\A\|^{ n - 3 - n\eps'}, 
\end{align*}  
the last inequality following from \eqref{Langlemma} with $\alpha_\nu = \alpha_{i_\nu}$.
Then, letting $\eps' = 3\eps /n$ and applying Lemma \ref{lem:5.4}, we conclude that
$$ 
\sum_{(u,v)\in\mathcal{C}} \frac{1}{(\N_k( F(u,v)))^{1/3}}  \ll \|\A\| \cdot \|\A\|^{(3 - n)/3 + n\eps'/3} = \|\A\|^{2 - n/3 + \eps},  
$$ 
which is satisfactory. This completes the proof of Theorem \ref{sumestimate}.

\begin{rem}
Let  $\gamma \in \left[ 1/3,1 \right]$.
The proof of Theorem \ref{sumestimate} can be  adapted to show 
that
$$
\|\A\|^{2 - \gamma n} \ll \sum_{\substack{ (u,v) \in \mathfrak{o}^2 \\ A_\nu \leq \sup\{ \|u\|_\nu, \|v\|_\nu\} < 2A_\nu \\  F(u,v) \neq 0}}   \frac{1}{(\N_k( F(u,v)))^{\gamma} } \ll\|\A\|^{2 - \gamma n + \eps}.
$$ 
The lower bound is trivial since $\N_k( F(u,v)) \ll \|\A\|^n$. This shows that our result is essentially best possible.
\end{rem}


\begin{thebibliography}{99}


\bibitem{B-M} V. Batyrev and Y. Manin,
Sur le nombre des points rationnels de hauteur born\'e des
vari\'et\'es alg\'ebriques. {\em  Math.\ Annalen}  {\bf
286} (1990), 27--43.


\bibitem{BomVal}
{ E. Bombieri and J. Vaaler},
{On Siegel's lemma. } {\em Invent.\ Math.} \textbf{73} (1983), 11--32.

\bibitem{BreBro}
{ R. de la Bret\`{e}che and T.D. Browning},
{Manin's conjecture for quartic del Pezzo surfaces with a conic fibration}.  {\em Duke Math.\ J.} \textbf{160}  (2011), 1--69.

\bibitem{Broberg}
{ N. Broberg},  {Rational points of cubic surfaces.} {\em Rational points on algebraic varieties. }
Progress in Math. {\bf 199}, Birkh\"{a}user, Basel, 13--35, 2001. 

\bibitem{Broberg2}
{ N. Broberg},  {Rational points on finite covers of $\mathbb{P}_1$ and $\mathbb{P}_2$.} {\em J.\ Number Theory} \textbf{101} (2003), 195--207.


\bibitem{Timbook}
{ T.D. Browning},  {\em Quantitative arithmetic of projective varieties.} 
Progress in Math. {\bf 277}, Birkh\"{a}user, Basel, 2009.

\bibitem{TimChatelet}
{ T.D. Browning},  {Linear growth for Ch\^{a}telet surfaces.}  {\em Math. Annalen} \textbf{346} (2010), 41--50.

\bibitem{TimRainer}
{ T.D. Browning and R. Dietmann},   {Solubility of Fermat equations.}  {\em Quadratic forms --- algebra, arithmetic and geometry}. Contemporary Mathematics \textbf{493}, 99--106, 2009.

\bibitem{BrowningHB}
{ T.D. Browning and  D.R. Heath Brown},  {Counting rational points on hypersurfaces.} 
{\em J.\ reine angew.\ Math.} \textbf{584} (2005), 83--115.

\bibitem{HB}
{ D.R. Heath-Brown,} { The density of rational points on cubic surfaces.} {\em Acta Arith.} \textbf{ 79} (1997), 17--30.


\bibitem{HB2} 
{ D.R. Heath-Brown},  {Counting rational points on cubic surfaces.}  {\em Ast\'{e}risque} {\bf 251} (1998), 13--30.

 \bibitem{Isko}
{ V.A. Iskovskih,}   {Birational properties of a surface of degree 4 in} $\mathbb{P}_k^4$. (Russian) 
{\em Mat.\ Sb.} {\bf 88} (1971), 31--37. 

\bibitem{Lang} { S. Lang},  {\em Fundamentals of Diophantine geometry.} Springer-Verlag, New York, 1983.


\bibitem{bonn}
P. Salberger, 
Counting rational points on projective varieties. {\em In preparation}, 2013.

\bibitem{Serre1} 
{ J.-P.  Serre},
{\em Lectures on the Mordell-Weil theorem}.  Third edition, Aspects of Mathematics, Friedr. Vieweg $\&$ Sohn, Braunschweig, 1997.

\bibitem{Shaf}
{ I. Shafarevich,}  {\em Basic algebraic geometry I.}  Second edition. Springer-Verlag, Berlin, 1994. 


\end{thebibliography}
\end{document}